\documentclass[12pt,oneside]{article}
\usepackage{amsmath,amssymb,amsfonts,amsthm}
\usepackage{eurosym}
\usepackage{amsfonts}
\usepackage{amsmath}
\usepackage{geometry}
\usepackage{mathtools}
\usepackage{pdfpages}
\usepackage{amsthm}
\usepackage{indentfirst}
\usepackage{hyperref}
\usepackage{cleveref}
\crefformat{section}{\S#2#1#3} 
\crefformat{subsection}{\S#2#1#3}
\crefformat{subsubsection}{\S#2#1#3}
\hypersetup{
	colorlinks=true,
	linkcolor=blue,
	filecolor=magenta,
	urlcolor=cyan,
	citecolor=red}

\newtheorem{theorem}{Theorem}[section]
\newtheorem{remark}[theorem]{Remark}
\newtheorem{lemma}[theorem]{Lemma}
\newtheorem{corollary}[theorem]{Corollary}
\newtheorem{definition}[theorem]{Definition}
\newtheorem{example}[theorem]{Example}
\newtheorem{proposition}[theorem]{Proposition}

\numberwithin{equation}{section}

\newcommand{\C}{{\mathcal C}}

\begin{document}
\title{\bf{ Anisotropic conformal change of conic pseudo-Finsler surfaces, II}}

\author{{\bf Nabil L. Youssef$\,^1$, S. G. Elgendi$^{2}$,   A. A. Kotb$^3$ and Ebtsam H. Taha$^1$ }}
\date{}

\maketitle
\vspace{-1.15cm}

\begin{center}
{$^1$Department of Mathematics, Faculty of Science,\\
Cairo University, Giza, Egypt}
\vspace{-8pt}
\end{center}

\begin{center}
 $^{2}$ Department of Mathematics, Faculty of Science,\\ Islamic University of Madinah,  Madinah, Kingdom of Saudia Arabia
\vspace{0.2cm}
\end{center}

\begin{center}
{$^3$Department of Mathematics, Faculty of Science,\\
Damietta University, Damietta , Egypt}
\vspace{-8pt}
\end{center}

\begin{center}
nlyoussef@sci.cu.edu.eg, nlyoussef2003@yahoo.fr\\
selgendi@iu.edu.sa, salahelgendi@yahoo.com\\
alikotb@du.edu.eg,  kotbmath2010@yahoo.com\\
ebtsam.taha@sci.cu.edu.eg, ebtsam.h.taha@hotmail.com

\end{center}

\begin{abstract}
This paper is a continuation of our investigation of the anisotropic conformal change of a conic pseudo-Finsler surface $(M,F)$, namely, the change $\overline{F}(x,y)=e^{\phi(x,y)}F(x,y)$ \cite{first paper}. We obtain the relationship between some important geometric objects of $F $ and their corresponding objects of $\overline{F}$, such as Berwald, Landsberg and Douglas tensors, as well as the T-tensor. In contrast to isotropic conformal transformation, under an anisotropic conformal transformation, we find out the necessary and sufficient conditions for a Riemannian surface to be anisotropically conformal transformed to Berwald or Landsberg or Douglas surfaces. Consequently, we determine under what condition the geodesic spray of a two-dimensional pseudo-Berwald metric $\overline{F}$ is Riemann metrizable by a two-dimensional pseudo-Riemannian metric $F$. We show an example of a conformal transformation of a Riemannian metric $F$ that is not geodesically equivalent to a Riemannian metric but is instead Berwaldian. Also, we determine the necessary and sufficient conditions for $F $ to be anisotropically conformally flat (i.e., $\overline{F}$ is Minkowskian). Moreover, we identify the required conditions for preserving the $T$-condition under an anisotropic conformal change. Finally, we establish the necessary conditions for a Riemannian metric to be anisotropically conformal to a Douglas metric.
\end{abstract}


\section*{Introduction}
A conic pseudo-Berwald space is a subclass of conic pseudo-Finsler spaces that admit a unique metric-compatible, torsion-free affine connection. In these spaces, the coefficients of the geodesic spray are quadratic in the directional argument $y$ or, equivalently,  the nonlinear (Barthel) connection coefficients are linear in $y$ \cite{Bacs-Matsumoto}. A conic pseudo-Finsler metric is said to be  conformally Berwald or conformally flat if it can be conformally changed to a Berwald or Minkowski metric, respectively. Under (an isotropic) conformal transformation, Hashiguchi \cite{Hashiguchi76} discussed some  special Finsler metrics such as  Minkowski, Landsberg and Berwald metrics. Furthermore, he found the conditions for a Finsler space to be conformal to one of these spaces. Under a non-homothetic  conformal transformation of Finsler surfaces, he found a necessary condition for the properties of being  Landsbergian, Minkowskian and Berwaldian to be preserved. The necessary and sufficient conditions for certain $(\alpha,\beta)$-metrics as well as  Finsler metrics to have a scalar field $S(x,y)$ satisfying either Kikuchi’s assumption or conditions weaker than Kikuchi’s assumption to be conformally Berwald have been found in  \cite{Hojo Matsumoto Okubo, Matsumoto 2001, Tamim-Youssef, Tayebi-Amini-Najafi}. The conditions under which Randers space is conformally flat have been obtained in \cite{Ichijyo-Hashiguchi}.  Additionally, the criteria for a three-dimensional Landsberg conformally flat space is Berwaldian, as has been determined in \cite{Prasad-Pandey-Singh}. The class of two-dimensional Finsler spaces exhibits unique geometric properties that distinguish it from  higher dimensional Finsler spaces \cite{An.In.Ma1993, Bacs-Matsumoto, Matsumoto 2003, shen2002,  Tayebi-Koh}.Studying anisotropic conformal transformation encounters some difficulties as the conformal factor depends not only on the position (as in conformal transformation) but also on the directional argument. In \cite{first paper}, the investigation of anisotropic conformal change of a conic pseudo-Finsler surface ($F(x,y)\longmapsto\overline{F}(x,y)=e^{\phi(x,y)}F(x,y)$) has been started.In this paper, we continue what we have started in \cite{first paper}. In ~\S 2,  we give some important technical lemmas that will be used in the sequel.  In $\S 3$, we find out the anisotropic transformation of Berwald curvature. Consequently, we get the required condition for F to be  anisotropically conformally Berwald. The Riemannian metric can be changed to any non-Riemannian conic pseudo-Finsler in an anisotropic way, as opposed to an isotropic way. For example, see Proposition \ref{first theorem of F bar berwald}.  Consequently, we discuss under what conditions the geodesic spray of $(M,\overline{F})$ coincides with the geodesic spray of a  Riemannian surface $(M,F)$, which is called the Riemann metrizability of the Berwald surface $(M,\overline{F})$.\par In the case of positive definite Finsler metrics, Szabó's metrization theorem \cite{szabo} states that every metric of Berwald type is Riemannian  metrizable. In other words, there exists a Riemannian metric on $M$ and the Berwald connection is exactly the Levi–Civita connection associated with this Riemannian metric. In \cite{Universe, J.Math.Phys, Voicu Elgendi}, the authors have investigated  the conditions required for extending Szabó’s metrizability theorem to Finsler spacetime. A new characterization of Finsler metrics of Berwald type has been introduced in \cite{Fuster}.   Under the anisotropic conformal change, we investigate under what conditions the geodesic spray of a  Berwald   metric $\overline{F}$ coincides with the geodesic spray of a  Riemannian metric $F$.  We provide an example in which  the geodesic spray of a  Berwald  metric $\overline{F}$ does not coincide with  the geodesic spray of a Riemannian metric $F$. This example is a counterexample to Szabó's metrization theorem \cite{szabo} of pseudo-Finsler surfaces.\par Additionally, in ~\S 4 we find the transformation of the Landsberg scalar and T-tensor. The $T$-condition refers to the vanishing of the $T$-tensor \cite{elgendi T}. Under an anisotropic conformal change, we find out the conditions that preserve the $T$-condition. Moreover, we investigate the conditions under which the conic pseudo-Finsler metric $\overline{F}$ is Landsbergian. Consequently, Proposition \ref{in view of unicorn} gives the conditions that make a Landsberg metric $\overline{F}$  Berwaldian.Finally, in ~\S 5 we examine the conditions under which $\overline{F}$ becomes either a Minkowski or Douglas metric under an anisotropic conformal transformation. We deduce that $F$ is anisotropically conformally flat if and only if  $F\partial_{i}\phi+\partial_{i}F=0.$ The Douglas tensor remains invariant under projective changes. Consequently, if \(F\) is Riemannian,  Berwald or projectively flat with $Q = 0$, then $\overline{F}$ is  a Douglas metric.  More generally, for a conic pseudo-Berwald Finsler surface, $\overline{F}$ is a Douglas metric if and only if $3\psi = \varepsilon\chi_{;2} + \mathcal{I}\chi$, where the functions \(\psi\) and \(\chi\) are explicitly defined in terms of \(P\), \(Q\), and their derivatives.\par We end the paper with a conclusion that highlights the present work. 

\section{Notation and Preliminaries}
Let $ M $ be a smooth manifold of dimension n and $ \pi : TM \longrightarrow M $ denote the canonical projection. Let $TM_0:=TM\setminus(0)$ be  the slit tangent bundle and  $(x^{i}, y^i)$ be the local coordinate system on $TM$.  
  The almost-tangent structure $J$ of $T M$ is the vector $1$-form   given  by $J = \frac{\partial}{\partial y^i} \otimes dx^i$. The globally defined  vertical vector field
$\C=y^i\frac{\partial}{\partial y^i}$ on $TM$ is known as 
the Liouville vector field. The notation  $C^{\infty} (TM_0)$  typically refers to the  set of smooth functions on $TM_0$. A function  $f \in C^{\infty} (TM_0)$ is said to be  positively homogeneous of degree $r$ in  the directional  argument $y$ (and  denoted by $h(r)$) if $$f(x,\lambda y)=\lambda^r f(x,y) \qquad \forall \lambda>0.$$

A spray on $M$ is a special vector field $S$ on $ T M$ such that $JS = \C$ and
$[\C, S] = S$. Locally, a spray is given by \cite{Gifone}
\begin{equation*}
  \label{eq:spray}
  S = y^i \frac{\partial}{\partial x^i} - 2G^i (x,y) \,\frac{\partial}{\partial y^i},
\end{equation*}
where $G^i(x,y)$ are the spray coefficients which are $h(2)$. A nonlinear (Ehresmann or Barthel) connection provides a smooth decomposition of the tangent bundle $TM_0$
 into horizontal $H(TM_0)$  and vertical $V(TM_0) $ subbundles. 
For each point $u=(x,y)$ on $TM_0$, the local basis of $V_u(TM_0)$ and $H_u(TM_0)$ are defined, respectively by,  $\dot{\partial}_i:=\dfrac{\partial}{\partial y^i}$   and $\delta_i:=\dfrac{\delta}{\delta x^i}=\dfrac{\partial}{\partial x^i}-G^j_i \dot{\partial}_j$, where $G^j_i=\dot{\partial}_i G^j$ are the coefficients of the Barthel connection. The Berwald connection and Berwald curvature are given, respectively, by  $G^h_{ij}:=\dot{\partial}_i G^h_j$ and $G^h_{ijk}:= \dot{\partial}_kG^h_{ij}.$ 

A conic sub-bundle of $TM$ is a non-empty open subset $\mathcal{A}  \subset TM_0$ which is invariant under scaling of tangent vectors by positive real numbers and  satisfies $ \pi(\mathcal{A}) = M. $ 

\begin{definition}\emph{\cite{Voicu Elgendi}}
  A conic pseudo-Finsler  metric on $ M $ is a smooth $h(1)$ function
  $\; F:\mathcal{A}\longrightarrow \mathbb{R}$ which satisfies for each point of $ \mathcal{A}$, the Finsler metric tensor $ g_{ij} (x,y)=\frac{1}{2}\dot{\partial}_{i}\dot{\partial}_{j}F^{2} (x,y)$ is non-degenerate. The pair $(M, F)$ is called a conic pseudo-Finsler manifold and $L=F^2$ is called the conic pseudo-Finsler Lagrangian.
 \end{definition} 

\par For a conic pseudo-Finsler surface, the angular metric matrix $(h_{ij})$ has rank one. Thus, a unique 1-form $m_i(x,y)$ can be chosen with  $\varepsilon=\pm1$ to express $h_{ij}$ \cite{An.In.Ma1993} as $h_{ij}=\varepsilon m_im_j.$
Since we have $g_{ij}=\ell_{i}\ell_{j}+h_{ij}$, where $\ell_i=\dot{\partial_{i}}F$ and $\ell^{i}=\frac{y^i}{F}$. Then, the Finsler metric tensor can be written as 
\begin{equation}\label{metric tensor of 2-dim}
g_{ij}=\ell_{i}\ell_{j}+\varepsilon m_im_j.
\end{equation}
Obviously,
\begin{equation}
\ell^{i}\ell_{i}=1\qquad \ell^{i}m_{i}=\ell_{i}m^{i}=0,\qquad m^{i}m_{i}=\varepsilon.
\end{equation} 
Hence $(\ell^i,m^i)$ is an orthonormal frame which is called a modified Berwald frame. Moreover, we have    $\mathfrak{g} =\det(g_{ij})=\varepsilon(\ell_1m_2-\ell_2m_1)^2$.  Hence, we can say that $\varepsilon$ is the signature of the metric.
Furthermore, by \cite[Lemma 2.7]{first paper}, we obtain
\begin{equation}\label{relation between m and ell}
(m^1,m^2)= \frac{\varepsilon}{\sqrt{\varepsilon \mathfrak{g}}}(-\ell_2,\ell_1),\qquad (m_1,m_2)= \sqrt{\varepsilon \mathfrak{g}} (-\ell^2,\ell^1),
\end{equation}
 It should be noted that there is a missing $\varepsilon$ in the formulae of $m^1$ and $m^2$ in \cite[Lemma 2.7]{first paper}.
 
   From \eqref{relation between m and ell} and $\delta_i F=\partial_i F-G^r_i\ell_r=0$, we get 
\begin{equation}\label{geodesic of 2 dim}
2G^i=y^r(\partial_r F)\ell^i+\frac{F^2(\dot{\partial_2}\partial_1 F-\dot{\partial_1}\partial_2 F)}{h}m^i,
\end{equation}
The main scalar $\mathcal{I}(x,y)$ is an $h(0)$  smooth function defined, from the Cartan tensor \cite{Berwald41}, by
\begin{equation}\label{cartan tensor Finsler surface} 
 F C_{ijk}=\mathcal{I}\;m_{i}m_{j}m_{k}.
\end{equation}
\begin{lemma}\label{properties.2.dim.Fins}\em\cite{Matsumoto 2003}
  \it Assume $(M,F)$ is a conic pseudo-Finsler surface. Then we get the following:
  \begin{description}
    \item[(a)] $F\dot{\partial _{j}}\ell_{i}=\varepsilon m_{i}m_{j}=h_{ij},\qquad F\dot{\partial _{j}}\ell^{i}=\varepsilon m^{i}m_{j} $,
    \item[(b)] $F\dot{\partial _{j}}m_{i}=-(\ell_{i}-\varepsilon \mathcal{I} m_{i})m_{j}$,\qquad $F\dot{\partial _{j}}m^{i}=-(\ell^{i}+\varepsilon \mathcal{I} m^{i})m_{j}$.
\end{description}
\end{lemma}
For  $f\in C^{\infty}(TM_0),$ we define the v-scalar derivatives $(f_{;1}, f_{;2})$  and 
h-scalar derivatives $(f_{,1}, f_{,2})$ with respect to $F$ by \cite[~\S 3.5]{An.In.Ma1993}:
\begin{equation}\label{vertical and horizontal deri of f}
f_{|i}=F\dot{\partial}_i f=f_{;1}\ell_{i}+f_{;2}m_{i} \qquad f|_i=\delta_i f=f_{,1}\ell_{i}+f_{,2}m_{i},
\end{equation}
where
\begin{align}
f_{;1}&=y^{i}\dot{\partial}_i f,\quad\qquad f_{;2}=\varepsilon F(\dot{\partial}_{i}f)m^{i},\label{comp of vertical  scalar deriv of f }\\
f_{,1}&=(\delta_{i}f)\ell^{i},\quad\qquad f_{,2}=\varepsilon (\delta_{i}f)m^{i}.\label{comp of horizontal scalar deriv of f}
\end{align}
\begin{definition}
    A function $f\in C^{\infty}(TM_0)$ is said to be horizontally constant if $\delta_{i}f=0,$ which is equivalent to $f_{,1}=f_{,2}=0$.
\end{definition}
Specifically, if $ f$ is $ h(r) $, then $f_{;1}=rf$. The commutation formulas  of $h(0)$  function $f$ are given by \cite{Berwald41}:
\begin{align}
 f_{,1,2}-f_{,2,1}=&-Rf_{;2},\label{first comutation}\\
 f_{,1;2}-f_{;2,1}=&f_{,2},\label{second comutation}\\
f_{,2;2}-f_{;2,2}=&-\varepsilon(f_{,1}+\mathcal{I}f_{,2}+\mathcal{I}_{,1}f_{;2}),\label{third comutation}
\end{align}
where $R$ is  the Gauss curvature or the $h$-scalar curvature.
\begin{definition}\emph{\cite{Bacs-Matsumoto}}
 The  $T$-tensor, Landsberg scalar, Berwald curvature and Douglas tensor of the conic pseudo-Finsler surface $(M,F)$ are given, respectively, by
\begin{align}
FT_{ijkh}=&\mathcal{I}_{;2} \,m_im_jm_km_h,\label{T tensor of F}\\
\mathcal{J}=&S(\mathcal{I})= F\,\mathcal{I}_{,1},\label{landesberg curvature for syrfaces}\\
F B^{i}_{jkr}=&[-2\mathcal{I}_{,1} \,\ell^{i}+\mathcal{I}_{2} \,m^{i}]m_{j}m_{k}m_{r},\label{berwald curvature 2}\\
3F D^{i}_{jkr}=&-(6\mathcal{I}_{,1}+\varepsilon\mathcal{I}_{2;2}+2\mathcal{I}\,\mathcal{I}_2)\ell^im_hm_jm_k,\label{douglas curvature}
\end{align}
where $\mathcal{I}_2=\mathcal{I}_{,1;2}+\mathcal{I}_{,2}.$
\end{definition}
Consequently, in a Finsler surface $(M,F)$ we have the following:
\begin{description}
\item[(a)]  $F$ is Landsbergian if and only if $\mathcal{I}_{,1}=0$.
\item[(b)] $F$ is Berwaldian if and only if $\mathcal{I}_{,1}=0$ and $\mathcal{I}_{,2}=0$.
\item [(c)]$F$ satisfies the $T$-condition ($F$ has a vanishing T-tensor) if and only if $\mathcal{I}_{;2}=0.$
\item[(d)]  $F$ is Douglasian if and only if $6\mathcal{I}_{,1}+\varepsilon\mathcal{I}_{2;2}+2\mathcal{I}\,\mathcal{I}_2=0.$ 
\end{description}
\begin{lemma}\label{homothety of phi}
Let $(M,F)$ be a two-dimensional conic pseudo-Finsler space  and $f$ be a smooth $h(0)$ function on $TM_0$. The function f is constant if one of the following conditions is satisfied:
\begin{description}
\item[(a)] $f_{,1}=f_{;2}=0$.
\item[(b)] $f_{,2}=f_{;2}=0$.
\item[(c)]$f$ is  a horizontally constant function provided that $R \neq 0.$
\end{description}
\end{lemma}
\begin{proof}
The proof is  direct consequence of \eqref{first comutation}, \eqref{second comutation} and \eqref{third comutation}.
\end{proof}

\begin{definition}\emph{\cite{first paper}}
The anisotropic conformal change of  a conic pseudo-Finsler metric $F $ is defined as
\begin{equation}\label{the anisotropic conformal transformation}
 F\longmapsto \overline{F}(x,y)=e^{\phi(x,y)}F(x,y), \quad F^{2}(\dot{\partial}_{i}\dot{\partial}_{j}\phi\
 +(\dot{\partial}_{i}\phi)(\dot{\partial}_{j}\phi)) m^{i}m^{j}+\varepsilon\neq0,
 \end{equation}
    given that the conformal factor $ \phi(x,y)$ is a smooth $h(0)$  function on $\mathcal{A}$.
In this case, we say that $F$ is anisotropically conformally changed to $\overline{F}$.
\end{definition}
In \cite{first paper}, we have discussed the anisotropic conformal change of a conic pseudo-Finsler surface $(M,F)$ equipped with a modified Berwald frame and determined how this change affects the components of the Berwald frame $(\ell,m)$ of $F$, that is,
\begin{align}
  \overline{\ell}_{i}&=e^{\phi}\,[\ell_{i}+\phi_{;2}\; m_{i}],\qquad
   \overline{\ell}^{i}=e^{-\phi}\,\ell^{i},\label{transform of ell^i}\\
  \overline{m}_{i}&=e^{\phi}\,\sqrt{\frac{\varepsilon}{\rho}}\,m_{i}, \qquad
  \overline{m}^{i}=e^{-\phi}\,\sqrt{\varepsilon\rho}\,[m^{i}-\varepsilon \phi_{;2}\;\ell^{i}],\label{transform of m^i}
  \end{align}
  where,
\begin{align}\label{formula of rho}
\rho&=\frac{1}{\sigma+\varepsilon-(\phi_{;2})^{2}},\qquad
\sigma=\phi_{;2;2}+\varepsilon \mathcal{I}\phi_{;2}+2\,(\phi_{;2})^{2}.
\end{align}
  Consequently, the anisotropic conformal transformation of the main scalar $\mathcal{I}$ is given by \cite{first paper} 
  \begin{align} 
  \overline{\mathcal{I}}&=(\varepsilon\rho)^{\frac{3}{2}} \,[\mathcal{I}(1+\varepsilon \sigma)+\frac{1}{2}\sigma_{;2}+\phi_{;2}(\sigma+2\varepsilon)],\label{transform of main scalar}\\
  &=\sqrt{\varepsilon\rho} \,[\mathcal{I}+2\varepsilon\phi_{;2}-\frac{\varepsilon}{2}(\ln \rho)_{;2}].\label{transform of main scalar 2}
  \end{align}
Furthermore, the anisotropic conformal change of the geodesic spray coefficients, Barthel connection and Berwald connection are given, respectively, by  
   \begin{align}
  \overline{G}^{i}&=G^i+Q\,m^i+P\ell^i,\label{transform of coefficient spray}\\
  \overline{G}_{j}^{i} &=G_{j}^{i}+\frac{1}{F}\left\lbrace 2P\ell^{i}\ell_{j}+(P_{;2}-Q)\ell^{i}m_{j}+2Q\ell_{j}m^{i}+(\varepsilon P+Q_{;2}-\varepsilon\mathcal{I}Q)m^{i}m_{j}\right\rbrace,\label{transform of Barthel connection}\\
  \overline{G}^{i}_{jk}& ={G}^{i}_{jk}+\frac{1}{F^{2}}[(2P\ell^{i}+2Qm^{i})\ell_{j}\ell_{k}+\{(P_{;2}-Q)\ell^{i}+(\varepsilon P+ Q_{;2}-\varepsilon\mathcal{I} Q)m^{i}\}\nonumber\\
 &\qquad(\ell_{j}m_{k}+\ell_{k}m_{j})+\{(\varepsilon P+P_{;2;2}-2Q_{;2}+\varepsilon\mathcal{I}P_{;2})\ell^{i}+(2\varepsilon P_{;2}\nonumber \\
 &\qquad+\varepsilon Q+ Q_{;2;2}-\varepsilon\mathcal{I}_{;2}Q-\varepsilon\mathcal{I} Q_{;2})m^{i}\}m_{j}m_{k}],\label{transform of Berwald connection}
\end{align}
  \begin{align}
2Q&=\varepsilon\rho F^2(\phi_{;2}\phi_{,1}+\phi_{,1;2}-2\phi_{,2}),\label{formula of Q only}\\
2P&=-\rho F^2\phi_{;2}(\phi_{;2}\phi_{,1}+\phi_{,1;2}-2\phi_{,2})+F^2\phi_{,1}.\label{formula of P only}
\end{align}
Clearly, from \eqref{formula of Q only} and \eqref{formula of P only}, we get
\begin{equation}\label{relation between P and Q}
2\varepsilon \phi_{;2}Q+2P=F^2\phi_{,1}.
\end{equation}
Since, any conic pseudo-Finsler metric is horizontally constant, we obtain
  \begin{equation}
\renewcommand{\arraystretch}{1.5}
  \left.\begin{array}{r@{\;}l}
(1)& F^2\ell^k\,\partial_k\phi=F^2\phi_{,1}+2G^k\phi_{;2}\;m_k,\quad
(2)\,F m^k\,\partial_k\phi=\varepsilon F\,\phi_{,2}+G^i_{k}\;\phi_{;2}\;m^k m_i,\\
(3)&\, \ell^k\,\partial_k F^2=4\,G^k\,\ell_{k},\qquad \qquad \qquad \,\,\,\,
(4) \,m^k\partial_k F^2=2FG_k^i\,\ell_{i}\,m^k.
  \end{array}\right\} \label{eq}
\end{equation}
\section{Technical Lemmas and Some Properties}
Assume that $(M,F)$ is a two-dimensional  conic pseudo-Finsler space equipped with the modified Berwald frame $(l,m)$. Under the anisotropic conformal transformation \eqref{the anisotropic conformal transformation}, $(l,m)$ is transformed to $(\overline{l},\overline{m})$. Consequently, one can assume that for all $f\in C^{\infty}(TM_0)$,   the v-scalar derivatives $(f_{;a} , f_{;b})$  and 
h-scalar derivatives $(f_{,a}, f_{,b})$ with respect to $\overline{F}$ are defined by
\begin{align}\label{vertical and horizontal deriv of scalar of F bar}
\overline{F}\dot{\partial}_i f=f_{;\,a}\overline{\ell}_{i}+f_{;\,b}\overline{m}_{i},  \qquad \overline{\delta}_i f=f_{,\,a}\overline{\ell}_{i}+f_{,\,b}\overline{m}_{i},
\end{align}
 where
\begin{align}
f_{;\,a}&=y^{i}\dot{\partial}_i f,\qquad \quad f_{;\,b}=\varepsilon \overline{F}(\dot{\partial}_{i}f)\overline{m}^{i},\label{comp of vertical scalar deriv of f bar}\\ 
f_{,\,a}&=(\overline{\delta}_{i}f)\overline{\ell}^{i}, \qquad \quad f_{,\,b}=\varepsilon (\overline{\delta}_{i}f)\overline{m}^{i}.\label{comp of horizontal scalar deriv of f bar}
\end{align}
\begin{proposition}\label{component of scalar derivative w.r.t F bar}
Let $\overline{F}=e^{\phi}F$ be the anisotropic conformal transformation \eqref{the anisotropic conformal transformation} of  $F$. The h- and v-scalar derivatives of a  function $f\in C^{\infty} (TM_0)$ satisfy the following properties:
\begin{align}
f_{;\,a}&=f_{;1},\label{f;1 transform}\\
f_{;\,b}&=\sqrt{\varepsilon\rho}\,( f_{;2}-\phi_{;2}f_{;1}),\label{f;2 transform}\\
f_{,\,a}&=e^{-\phi} \left[f_{,1}-\frac{2}{F^2}(Pf_{;1}+\varepsilon Qf_{;2})\right],\label{f,1 transform}\\
f_{,\,b}&= e^{-\phi}\sqrt{\varepsilon\rho}\left[f_{,2}-\phi_{;2}f_{,1}-\frac{1}{F^2}\lbrace (P_{;2}-Q-2\phi_{;2}P)f_{;1}+\varepsilon(\varepsilon P+Q_{;2}-\varepsilon\mathcal{I}Q-2\phi_{;2}Q)f_{;2}\rbrace\right].\label{f,2 transform}
\end{align}
\end{proposition}
\begin{proof}
From \eqref{comp of vertical  scalar deriv of f } and \eqref{comp of vertical scalar deriv of f bar}, we get $f_{;\,a}=f_{;1}.$\\
By  \eqref{comp of vertical  scalar deriv of f } and \eqref{comp of vertical scalar deriv of f bar}, we have
\begin{equation*}
f_{;\,b}=\varepsilon \overline{F}\;(\dot{\partial}_if)\;\overline{m}^{i}
=\sqrt{\varepsilon\rho}[\varepsilon F(\dot{\partial}_if){m}^{i}-\phi_{;2}F(\dot{\partial}_if)\ell^i]
=\sqrt{\varepsilon\rho}\left(  f_{;2}-\phi_{;2}f_{;1}\right).
\end{equation*}
Also, from \eqref{comp of vertical  scalar deriv of f },  \eqref{comp of horizontal scalar deriv of f}, \eqref{transform of coefficient spray} and \eqref{comp of horizontal scalar deriv of f bar},   we obtain
\begin{align*}
f_{,\,a}=&\frac{y^i}{\overline{F}}\,\overline{\delta}_{i}f
=\frac{e^{-\phi}}{F} \,[y^i\partial_{i}f-2\overline{G}^i\dot{\partial}_i f]
=\frac{e^{-\phi}}{F} \,[y^i\partial_{i}f-2{G}^i\dot{\partial}_i f-2P\ell^{i}\dot{\partial}_i f-2Qm^{i}\dot{\partial}_i f]\\
=&e^{-\phi} \left[ f_{,1}-\frac{2}{F^2}(Pf_{;1}+\varepsilon Qf_{;2})\right].
\end{align*}  
Finally, from  \eqref{comp of vertical  scalar deriv of f }, \eqref{comp of horizontal scalar deriv of f},  \eqref{transform of Barthel connection} and \eqref{comp of horizontal scalar deriv of f bar},  we get
\begin{align*}
f_{,\,b}=&\varepsilon \,\overline{m}^{i}\,\overline{\delta}_{i}f
=\varepsilon \,e^{-\phi}\,\sqrt{\varepsilon\rho}\,({m}^{i}-\varepsilon\phi_{;2}\ell^i) [\partial_{i}f-\overline{G}^r_{i}\,\dot{\partial}_r f]\\
= &e^{-\phi}\sqrt{\varepsilon\rho}\left[  f_{,2}-\phi_{;2}\,f_{,1}-\frac{1}{F^2}\left[    (P_{;2}-Q-2\phi_{;2}\,P)f_{;1}+\varepsilon(\varepsilon P+Q_{;2}-\varepsilon\mathcal{I}\,Q-2\phi_{;2}\,Q)f_{;2}\right]  \right] .
\end{align*}
\vspace*{-1.1cm}\[\qedhere\]
\end{proof}

\begin{lemma}\label{formula of I;a I,a I,b}
Under the  anisotropic conformal transformation  \eqref{the anisotropic conformal transformation},  the     h-scalar and v-scalar derivatives of the main scalar  $\overline{\mathcal{I}}$ of  $\overline{F}$ are given by:
\begin{align}
\overline{\mathcal{I}}_{;\,b}&=\sqrt{\varepsilon\rho} \;\overline{\mathcal{I}}_{;2},\label{I;b bar formula}\\
\overline{\mathcal{I}}_{,\,a}&=e^{-\phi} \,[\overline{\mathcal{I}}_{,1}-\frac{2\varepsilon}{F^2} Q\,\overline{\mathcal{I}}_{;2}],\label{I,a bar formula}\\
\overline{\mathcal{I}}_{,\,b}&= e^{-\phi}\sqrt{\varepsilon\rho}\;[\overline{\mathcal{I}}_{,2}-\phi_{;2}\,\overline{\mathcal{I}}_{,1}-\frac{\varepsilon}{F^2}(\varepsilon P+Q_{;2}-\varepsilon\mathcal{I}\,Q-2\phi_{;2}\,Q)\overline{\mathcal{I}}_{;2}]\label{I,b bar formula},
\end{align}
where $\overline{\mathcal{I}}_{;2}, \;\overline{\mathcal{I}}_{,1}, \text{and }\overline{\mathcal{I}}_{,2}$ are the components of the v-scalar and h-scalar derivatives of the main scalar $\overline{\mathcal{I}}$ with respect to $F$. Moreover, we have: 
\begin{align}
\overline{\mathcal{I}}_{;2}&=\frac{\sqrt{\varepsilon\rho}}{2\rho}\,[\rho_{;2}(\mathcal{I}+2\varepsilon\,\phi_{;2}+\frac{\varepsilon}{2}(\ln \rho)_{;2})+2\rho(\mathcal{I}_{;2}+2\varepsilon\phi_{;2;2})-\varepsilon\rho_{;2;2}].\label{I; 2 bar formula}\\
\overline{\mathcal{I}}_{,1}&=\frac{\sqrt{\varepsilon\rho}}{2\rho}\,[\rho_{,1}(\mathcal{I}+2\varepsilon\phi_{;2}+\frac{\varepsilon}{2} (\ln \rho)_{;2})+2\rho(\mathcal{I}_{,1}+2\varepsilon\phi_{;2,1})-\varepsilon\rho_{;2,1}].\label{I,1 bar formula}\\
\overline{\mathcal{I}}_{,2}&=\frac{\sqrt{\varepsilon\rho}}{2\rho}\,[\rho_{,2}(\mathcal{I}+2\varepsilon\phi_{;2}+\frac{\varepsilon}{2}(\ln \rho)_{;2})+2\rho(\mathcal{I}_{,2}+2\varepsilon\phi_{;2,2})-\varepsilon\rho_{;2,2}].\label{I,2 bar formula}
\end{align}
\end{lemma}
\begin{proof}
It follows by applying  Proposition \ref{component of scalar derivative w.r.t F bar} to the formula \eqref{transform of main scalar 2} of $\mathcal{\overline{I}}$ taking into account that $\mathcal{\overline{I}}$ is $h(0).$ 
\end{proof}

\begin{remark}\label{remark rho_,1 and rho_,2}
\begin{description}
\item[(a)]From \eqref{f;1 transform}, the v-scalar  derivative of a function $f\in C^{\infty} (TM_0)$ in the direction $\ell_i$  is anisotropically  conformally invariant, i.e. , $f_{;a}=f_{;1}.$
\item[(b)]From \eqref{f;2 transform},  the v-scalar derivative of a function $f\in C^{\infty} (TM_0)$ in the direction $m_i$  is $\sqrt{\varepsilon\rho}$ invariant under the anisotropic conformal transformation \eqref{the anisotropic conformal transformation} if and only if either the conformal factor is a function of $x$ only or $f$ is an $h(0)$ function.
\item[(c)]From  \em{\eqref{formula of rho}}, we get
\begin{align*}
\rho_{,1}=&-\rho^2\,(\phi_{;2;2,1}+\varepsilon\mathcal{I}_{,1}\,\phi_{;2}+\varepsilon\mathcal{I}\,\phi_{;2,1}+2\phi_{;2}\,\phi_{;2,1}).\\
\rho_{,2}=&-\rho^2(\phi_{;2;2,2}+\varepsilon\mathcal{I}_{,2}\,\phi_{;2}+\varepsilon\mathcal{I}\,\phi_{;2,2}+2\phi_{;2}\,\phi_{;2,2}).\\
\rho_{;2,1}=&-2\rho\rho_{,1}(\phi_{;2;2;2}+\varepsilon\mathcal{I}_{;2}\phi_{;2}+\varepsilon\mathcal{I}\phi_{;2;2}+2\phi_{;2}\,\phi_{;2;2})-\rho^2(\phi_{;2;2;2,1}+\varepsilon\mathcal{I}_{;2,1}\,\phi_{;2}\\
&+\varepsilon\mathcal{I}_{;2}\,\phi_{;2,1}+\varepsilon\mathcal{I}_{,1}\,\phi_{;2;2}+\varepsilon\mathcal{I}\,\phi_{;2;2,1}+2\phi_{;2,1}\,\phi_{;2;2}+2\phi_{;2}\,\phi_{;2;2,1}).
\end{align*}
Similarly, we can get $\rho_{;2},$  $\rho_{;2;2}$ and $\rho_{;2,2}$.
\item[(d)]Under the anisotropic conformal transformation \eqref{the anisotropic conformal transformation}, the condition $\sigma=(\phi_{;2})^2$ mentioned in  \cite[Theorem 2.11]{first paper} is equivalent to the property $ \overline{\mathfrak{g}}=e^{4\phi}\mathfrak{g}$,  where $\mathfrak{g}=\det(g_{ij}).$
\end{description}
\end{remark}

\begin{lemma}\label{proposition phi,phi_;2 and phi_;2;2 d_h closed}
Let  $F$ be a Landsberg surface  and f be an $h(0)$ smooth function on $TM_0$. If $f$ is horizontally constant, then $f_{;2}$ and $f_{;2;2}$ are horizontally constant.
\end{lemma}
\begin{proof}
Let $f$ be horizontally constant and $F$  a Landsberg metric. 
Applying  \eqref{second comutation} and \eqref{third comutation} to $f$, we obtain $f_{;2,2}=f_{;2,1}=0,$ which means that $f_{;2}$ is horizontally constant.
Similarly, applying \eqref{second comutation} and \eqref{third comutation} for $f= f_{;2}$,  we get $f_{;2;2,2}=f_{;2;2,1}=0,$
that is,  $f_{;2;2}$ is horizontally constant. 
\end{proof}
Using Remark \ref{remark rho_,1 and rho_,2} and Lemma \ref{proposition phi,phi_;2 and phi_;2;2 d_h closed}, where $f=\phi_{;2}$, we get the following result.
\begin{corollary}\label{proposition phi_;2  d_h closed and F riemannian}
Let  $(M,F)$ be a Riemannian  surface  and \eqref{the anisotropic conformal transformation}  be the anisotropic conformal transformation  provided that, $\phi_{;2}$ is horizontally constant. Then $\rho$ and $\rho_{;2}$ are horizontally constant. That is, $\rho_{,1}=\rho_{,2}=\rho_{;2,1}=\rho_{;2,2}=0$.
\end{corollary}
\begin{definition}
A smooth  function $f$ on  $ TM_0$ is called a first integral of a geodesic spray  $S$ if $S(f)=0.$ 
\end{definition}
In our context, a smooth function $f$ is a first integral of the geodesic spray $S$ if and only if $f_{,1}=0$, this is because $S(f)=y^i\delta_if=F\ell^i\delta_if=F f_{,1}=0$. 
\begin{remark}
   Let $f$ be  a smooth  function on  $ TM_0$ and \eqref{the anisotropic conformal transformation} be the anisotropic conformal transformation of a conic pseudo-Finsler metric $F$. Then, we have the following assertions:
\begin{description}
    \item[(a)] The function $f$ is a first integral of the geodesic spray $\overline{S}$  if and only if $F S(f)=2(f_{;1}P+\varepsilon f_{;2} Q )$.
    \item[(b)] The property that $f$ is a first integral of the geodesic spray is $e^{-\phi}$ invariant ($f_{,a}=e^{-\phi}f_{,1}$) if and only if $f_{;1}P+\varepsilon f_{;2} Q=0$.
    \item[(c)] The property that an  $h(0)$ function  $f$ is a first integral of the geodesic spray is $e^{-\phi}$ invariant if and only if either $f$ is a function of $x $ only or the two sprays $S$ and $\overline{S}$ are projectively equivalent.
    \item[(d)]Assume that $f$ is an $h(0)$ function  of $y$ only, then  $ f_{,1}\ell_i+f_{,2}m_i=\delta_{i}f=-\frac{f_{;2}}{F}G^r_im_r.$ \\
    Consequently,  multiplying by $\ell^i$, the function $f$ is a first integral of $S$ if and only if either $f_{;2}=0$ or $G^rm_r=0$, which is equivalent to $f$ is either a constant function or the Finsler metric $F$ is projectively flat \emph{\cite[Remark 4.6]{first paper}}.
\end{description} 
\end{remark}
Let us end this section by introducing the following  identity.
\begin{lemma}\label{lemma of identites}
 Let $(M,F)$ be a conic pseudo-Finsler surface. Under the anisotropic conformal transformation \eqref{the anisotropic conformal transformation}, we have the following identity:
 \begin{equation}
\phi_{;2}\,P+P_{;2}+\varepsilon \phi_{;2}\,Q_{;2} -(\mathcal{I}\,\phi_{;2}+1)\,Q-F^2\,\phi_{,2}=0.\label{identity 1}
\end{equation}
\end{lemma}
\begin{proof}
Applying $(;2)$ to \eqref{relation between P and Q}, we obtain
\begin{equation}\label{formula P_;2}
    P_{;2}=-\varepsilon\phi_{;2;2}\,Q-\varepsilon\,\phi_{;2}\,Q_{;2}+\frac{F^2}{2}\,\phi_{,1;2}.
\end{equation}
By using the expression of $P$ in \eqref{relation between P and Q}   and  \eqref{formula P_;2}, we get 
\begin{align*}
    P_{;2}+\phi_{;2}\,P+\varepsilon \,\phi_{;2}\,Q_{;2}-\mathcal{I}\,\phi_{;2}\,Q -F^2\,\phi_{,2}-Q=&-\varepsilon\phi_{;2;2}\,Q+\frac{F^2}{2}\phi_{,1;2}+\frac{F^2}{2}\phi_{;2}\,\phi_{,1}-\varepsilon (\phi_{;2})^2\,Q\\
    &-\mathcal{I}\phi_{;2}\,Q -F^2\,\phi_{,2}-Q\\
    =&-\varepsilon \,Q(\phi_{;2;2}+(\phi_{;2})^2+\varepsilon\,\mathcal{I}\,\phi_{;2}+\varepsilon)\\
    &+\frac{F^2}{2}(\phi_{;2}\,\phi_{,1}+\phi_{,1;2}-2\phi_{,2}).
\end{align*}
From \eqref{formula of rho} and \eqref{formula of Q only}, we deduce
$
    P_{;2}+\phi_{;2}\,P+\varepsilon\, \phi_{;2}\,Q_{;2} -\mathcal{I}\,\phi_{;2}\,Q-F^2\,\phi_{,2}-Q
    =-\frac{\varepsilon}{\rho} Q+\frac{\varepsilon}{\rho}Q=0.
$
\end{proof}
\section{ Anisotropically Conformally Berwald }
Our next goal is to show  how the anisotropic conformal transformation \eqref{the anisotropic conformal transformation} can transform a Riemannian metric  $F$ to a Finsler metric  $\overline{F}$. Specifically, we discuss  the case where a pseudo-Riemannian surface is anisotropically conformally transformed to a conic pseudo-Berwald surface. Consequently, we study the property that the conic pseudo-Berwald $\overline{F}$ is Riemann metrizable by $F$. Furthermore, we find under what conditions the conic-pseudo Finsler $F$ is anisotropically conformally changed to a conic pseudo-Berwald $\overline{F}.$
\begin{definition}
A Finsler space $(M,F)$ is said to be  anisotropically conformally Berwald if  $F$ is anisotropically conformally related to a Berwald metric  $\overline{F}.$
\end{definition}
\begin{lemma}
 Under the anisotropic conformal change \eqref{the anisotropic conformal transformation},   the components of the Berwald curvatures  $\overline{B}^{i}_{jkr}$  and $B^{i}_{jkr}$ are related by
\begin{flalign}\label{transform of berwald curvature of spray derivative}
  F^3\overline{B}^{i}_{jkr}  =&F^3 B^{i}_{jkr}+ [\{ (-3+\mathcal{I}_{;2})\varepsilon\,Q+(1+\mathcal{I}_{;2}+2\varepsilon\mathcal{I}^{2})\varepsilon\,P_{;2}+P_{;2;2;2}-3Q_{;2;2}\nonumber \\ &+3\varepsilon\mathcal{I}P_{;2;2}-3\varepsilon\mathcal{I}Q_{;2}+2\mathcal{I}P\}\ell^{i}+\{3P-(\mathcal{I}+\varepsilon\mathcal{I}_{;2;2}+\mathcal{I}\mathcal{I}_{;2})Q\\
 &-(2\varepsilon\mathcal{I}_{;2}+\mathcal{I}^{2}-\varepsilon) Q_{;2}+3\varepsilon P_{;2;2}+Q_{;2;2;2}+3\mathcal{I}P_{;2}\}m^{i}]m_{j}m_{k}m_{r}.\nonumber
\end{flalign}
\end{lemma}
\begin{proof}
This formula can be obtained by multiplying \eqref{transform of Berwald connection} by $F$, then differentiating with respect to $y^{r}$ and using  Lemma \ref{properties.2.dim.Fins}.
\end{proof}

\begin{remark}
\begin{description}
    \item[(a)] The property  of a conic pseudo-Finsler surface being Berwaldian is preserved under the anisotropic conformal change \eqref{the anisotropic conformal transformation}, is equivalent to
    \begin{equation*}
(-3+\mathcal{I}_{;2})\varepsilon\,Q-3\varepsilon\mathcal{I}Q_{;2}-3Q_{;2;2}+2\mathcal{I}P+ (1+\mathcal{I}_{;2}+2\varepsilon\mathcal{I}^{2})\varepsilon\,P_{;2}+3\varepsilon\mathcal{I}P_{;2;2}+P_{;2;2;2}=0, 
\end{equation*}
\begin{equation*}
   -(\mathcal{I}+\varepsilon\mathcal{I}_{;2;2}+\mathcal{I}\mathcal{I}_{;2})Q
 -(2\varepsilon\mathcal{I}_{;2}+\mathcal{I}^{2}-\varepsilon) Q_{;2}+Q_{;2;2;2}+3P+3\mathcal{I}P_{;2}+3\varepsilon P_{;2;2}=0. 
\end{equation*}
\item[(b)]  If $P=Q=0$ (which is equivalent to  the anisotropic conformal factor   $\phi$ being  horizontally constant), then property  of $F$ being  Berwaldian is invariant.
\end{description}
\end{remark}

Under the anisotropic conformal change \eqref{the anisotropic conformal transformation}, we can determine the conditions under which a Riemannian surface is anisotropically conformally Berwald, in the following results:

 \begin{proposition}\label{first theorem of F bar berwald}
Assume that the  conic pseudo-Finsler  surface $(M,F)$ is anisotropically related  to $(M,\overline{F})$. If  $ F$ is  Riemannian, then $\overline{F}$ is Berwaldian under the conditions that  $$P +\varepsilon P_{;2;2}=0, \quad \quad Q +\varepsilon  Q_{;2;2}=0.$$ 
\end{proposition}
 \begin{proof}
    If $F $ is a Riemannian metric, then we have $\mathcal{I}=0,\,\, B^{i}_{jkr}=0.$
Thus, \eqref{transform of berwald curvature of spray derivative} becomes
 \begin{equation}\label{Berwald equivalence}
 \overline{B}^{i}_{jkr}  =\frac{1}{F^{3}}[ \{ -3\varepsilon Q+\varepsilon P_{;2}+P_{;2;2;2}-3Q_{;2;2}\}\ell^{i}+\{3P+\varepsilon Q_{;2}+3\varepsilon P_{;2;2}+Q_{;2;2;2}\}m^{i}]m_{j}m_{k}m_{r}.
 \end{equation}
 Hence, $ \overline{F}$ is Berwaldian  if  $P +\varepsilon P_{;2;2}=0$ and $Q +\varepsilon Q_{;2;2}=0.$ 
 \end{proof}
 
 \begin{theorem}\label{first theorem of F bar berwald 2}
Let $(M,F)$ be  conic pseudo-Riemannian surface related to $(M,\overline{F})$ by  the anisotropic conformal transformation \eqref{the anisotropic conformal transformation}. Then  $\overline{F}$ is Berwaldian if and only if
\begin{equation}\label{equivalence for berwald}
    -3\varepsilon Q-3Q_{;2;2}+\varepsilon P_{;2}+P_{;2;2;2}=0,\;\; \quad  \quad \;\;3P+3\varepsilon P_{;2;2}+\varepsilon Q_{;2}+Q_{;2;2;2}=0.
\end{equation}  
Additionally, whenever $\overline{S}$ and $S$ are projectively equivalent,  $ \overline{F}$ is Berwaldian  if and only if $$\phi_{,1}+\varepsilon\phi_{,1;2;2}=0.$$
\end{theorem}
 \begin{proof}
 Assume that $F $ is Riemannian.  From  \eqref{Berwald equivalence}, it is clear that $ \overline{B}^{i}_{jkr}  =0$ if and only if $P$ and $Q$ satisfy \eqref{equivalence for berwald}.
\par Assume $S $ and $\overline{S}$ are projectively equivalent, that is, $Q=0$ and  $P=\frac{1}{2}F^2\phi_{,1}$ (by \cite[Theorem 4.2]{first paper}).
 Hence,   \eqref{equivalence for berwald}  is simplified to be $P +\varepsilon P_{;2;2}=0$.  Since $P=\frac{1}{2}F^2\phi_{,1}$, then  $\overline{F}$ is Berwaldian if and only if   $ \phi_{,1}+\varepsilon\phi_{,1;2;2}=0$. 
 \end{proof}
 
\begin{proposition}
Under the anisotropic conformal transformation \eqref{the anisotropic conformal transformation}, if the conformal factor is first integral of the geodesic spray $S$ and the two sprays $S$ and $\overline{S}$ are projectively equivalent, then the property of Berwaldian is preserved.
\end{proposition}
\begin{proof}
The conformal factor $ \phi$ is a first integral of the geodesic spray $S$ means that $F\phi_{,1}=0$.
Thereby, \eqref{relation between P and Q} becomes
\begin{equation}\label{2nd eq 1st inteqral}
\varepsilon \phi_{;2}Q+P=0.
\end{equation}
Since, $S$ and $\overline{S}$ are projectively equivalent,  i.e.,  $Q=0$,  then by \eqref{2nd eq 1st inteqral}, we get
$P=0.$ Therefore, $\overline{S}=S$ by \eqref{transform of coefficient spray}.
Since the geodesic spray is anisotropically conformally invariant and  $F$ is  Berwaldian then $\overline{F}$ is  Berwaldian.  In this case, we conclude that Berwald spaces are preserved under anisotropic conformal transformation.   
\end{proof}

 Under the anisotropic conformal transformation, we find the required conditions for a conic pseudo-Finsler surface to be anisotropically conformally changed to  a conic pseudo-Berwald surface.

\begin{proposition}If $R\neq0$ and $\phi_{;2}$ is horizontally constant, 
then the property of $F$ being Berwaldian is conserved under the anisotropic conformal transformation \eqref{the anisotropic conformal transformation}.
\end{proposition}
\begin{proof}
Let $F$ be  a Berwald metric, that is, $\mathcal{I}_{,1}=0=\mathcal{I}_{,2}$. Applying \eqref{second comutation} for $f=\mathcal{I}$, we get
\begin{align}\label{1st relation of prop of berwald form phi_;2 d_h closed}
\mathcal{I}_{;2,1}=0.
\end{align}
Assume that $\phi_{;2}$ is horizontally constant, i.e., 
$\phi_{;2,1}=\phi_{;2,2}=0$. 
Consequently, by  Lemma \ref{proposition phi,phi_;2 and phi_;2;2 d_h closed}, we have $\phi_{;2;2}$ and $\phi_{;2;2;2}$ are horizontally constant, that is, 
$\phi_{;2;2,1}=\phi_{;2;2,2}=\phi_{;2;2;2,1}=\phi_{;2;2;2,2}=0.$
Therefore, from Remark \ref{remark rho_,1 and rho_,2} (c), we get
\begin{align}\label{3rd relation of prop of berwald form phi_;2 d_h closed}
\rho_{,1}=\rho_{,2}=\rho_{;2,1}=\rho_{;2,2}=0.
\end{align}
Thus, by  \eqref{I,1 bar formula}  and \eqref{I,2 bar formula}, we have $
 \mathcal{\overline{I}}_{,1}=\mathcal{\overline{I}}_{,2}=0. $
Applying \eqref{third comutation} for $f=\overline{\mathcal{I}}$, (since $R\neq0$), we get
$\overline{\mathcal{I}}_{;2}=0.$ Hence,  $
 \mathcal{\overline{I}}_{,a}=\mathcal{\overline{I}}_{,b}=0$, by \eqref{I,a bar formula} and \eqref{I,b bar formula}, which means $\overline{F}$ is a Berwald metric.     
\end{proof}
 According to  \cite[Eqn.(20)]{Fuster} using a pseudo-Riemannian Lagrangian $A(x,y)=a_{ij}(x)y^iy^j,$  they found a necessary and  sufficient condition  for $A$ to be anisotropically conformally transformed to a conic pseudo-Berwald Lagrangian $L$, via, $L=\Omega \, A,$ where $\Omega$ is  the anisotropic conformal factor which is a smooth $h(0)$ function on $\mathcal{A}$. This condition is equivalent to $A\,\delta_i\Omega = y^k\, M^j_{ik} \,\dot{\partial}_jL$ for any smooth symmetric tensor $M^j_{ik}(x).$ 
  In our context, under the anisotropic conformal transformation \eqref{the anisotropic conformal transformation} provided that $F$ is a conic pseudo-Riemannian metric, the necessary  and sufficient conditions for $\overline{F}$ to be a conic pseudo-Berwald metric  is the existence of a symmetric tensor $M^i_{jk}$ of type $(1,2)$ on the base manifold $M$  such that
\begin{equation}\label{fuster equivelent}
    F^2 \,\delta_{i}(e^{2\phi})=y^k\,M^j_{ik} \,\dot{\partial}_{j}\overline{F}^2
\end{equation} 
 \begin{theorem}\label{theorem of existence symmetric tensor on M}
   Let $F$ be a conic pseudo-Riemannian  metric  anisotropically  conformally changed to $\overline{F}$ by the anisotropic conformal transformation \eqref{the anisotropic conformal transformation}. The following assertions are equivalent:
   \begin{description}
       \item[(a)]The functions $P$ and $Q$, given by \eqref{formula of P only} and \eqref{formula of Q only}, are satisfy 
\begin{equation}   \label{PQ}
       -3\varepsilon Q-3Q_{;2;2}+\varepsilon P_{;2}+P_{;2;2;2}=0, \quad \quad 3P+3\varepsilon P_{;2;2}+\varepsilon Q_{;2}+Q_{;2;2;2}=0.
       \end{equation}    
       \item[(b)] There exists a symmetric tensor on the base manifold $M$ given by
        \begin{align}\label{symmetric tensor on the base maifold first}
     M^i_{jk}=& \frac{1}{F^{2}}[2(P\ell^{i}+Qm^{i})\ell_{j}\ell_{k}+\{(P_{;2}-Q)\ell^{i}+(\varepsilon P+ Q_{;2})m^{i}\} (\ell_{j}m_{k}+\ell_{k}m_{j})\nonumber\\
     &+\{(\varepsilon P+P_{;2;2}-2Q_{;2})\ell^{i}+(2\varepsilon P_{;2}+\varepsilon Q+ Q_{;2;2})m^{i}\}m_{j}m_{k}].
  \end{align}
    \item[(c)]$\overline{F}$ is Berwaldian.  
   \end{description}
\end{theorem}
\begin{proof}
$\textbf{(a)}\Longrightarrow \textbf{(b)}$ Let $(M,F)$ be a conic pseudo-Riemannian surface such that the functions $P$ and $Q$,  determined by \eqref{formula of P only} and \eqref{formula of Q only},  satisfy  \eqref{PQ}. Then,  by Theorem \ref{first theorem of F bar berwald 2},   $(M,F)$ is Berwaldian. If $G^i_{jk}$ and $\overline{G}^i_{jk}$ are the Berwald connection coefficients of $F$ and $\overline{F}$, respectively, then the difference between them is a symmetric tensor of type $(1,2)$ say $M^i_{jk}$.
From \eqref{transform of Berwald connection} and $F$ being Riemannian metric, the formula of the tensor $M^i_{jk}$ is given by \eqref{symmetric tensor on the base maifold first}. 
 Clearly, $ M^i_{jk}$  is  symmetric  in $j,\;k$.
\par   Now, showing that $\dot{\partial}_r M^i_{jk}=0$  is sufficient to complete the proof of $\textbf{(a)}\Longrightarrow \textbf{(b)}$.  Since  $F$ is a conic pseudo-Riemannian metric, Lemma \ref{properties.2.dim.Fins} gives raise to 
     $$F\dot{\partial _{j}}\ell_{i}=\varepsilon m_{i}m_{j},\quad F\dot{\partial _{j}}\ell^{i}=\varepsilon m^{i}m_{j}, \quad F \dot{\partial_{j}}m_{i}=-\ell_{i}m_j,\qquad F\dot{\partial _{j}}m^{i}=-\ell^{i}m_j.$$ After some slightly tedious  straightforward calculations, we conclude that all terms of $\dot{\partial}_r M^i_{jk}$  vanish except the following two terms
 \begin{align*}
        \dot{\partial}_r M^i_{jk}=\frac{1}{F^3}[(-3\varepsilon Q-3Q_{;2;2}+\varepsilon P_{;2}+P_{;2;2;2})\ell^i + (3P+3\varepsilon P_{;2;2}+\varepsilon Q_{;2}+Q_{;2;2;2})m^i]\,m_jm_km_r.
 \end{align*}
The symmetric tensor $M^i_{jk}$ does not depend on the directional argument since $$-3\varepsilon Q-3Q_{;2;2}+\varepsilon P_{;2}+P_{;2;2;2}=0, \qquad \qquad 3P+3\varepsilon P_{;2;2}+\varepsilon Q_{;2}+Q_{;2;2;2}=0.$$ 
$\textbf{(b)}\Longleftrightarrow \textbf{(c)}$ It suffices to prove that \eqref{symmetric tensor on the base maifold first} satisfies \eqref{fuster equivelent}. Thereby,
plugging \eqref{symmetric tensor on the base maifold first} into \eqref{fuster equivelent}, we get
\begin{align*}
   2e^{2\phi} F^2(\phi_{,1}\ell_{i}+\phi_{,2}m_i)=&\frac{1}{F^{2}}[(2P\ell^{j}+2Qm^{j})\ell_{i}\ell_{k}+\{(P_{;2}-Q)\ell^{j}+(\varepsilon P+ Q_{;2})m^{j}\} (\ell_{i}m_{k}+\ell_{k}m_{i})\\
     &+\{(\varepsilon P+P_{;2;2}-2Q_{;2})\ell^{j}+(2\varepsilon P_{;2}+\varepsilon Q+ Q_{;2;2})m^{j}\}m_{i}m_{k}][2e^{2\phi}F^2\ell_j\ell^k\\
     &+2e^{2\phi}F^2\phi_{;2}\ell^km_j].
\end{align*}
\par Using $\ell_im^i=0,$ we obtain $
   F^2(\phi_{,1}\ell_{i}+\phi_{,2}m_i) =\left\{2P+2\varepsilon \phi_{;2}Q\right\}\ell_i+\left\{(P_{;2}-Q)+\phi_{;2}(P+\varepsilon Q_{;2})\right\} m_i. $
Consequently, by contracting with $\ell^i$ and $m^i$, we deduce 
$$
 F^2\phi_{,1}=2P+2\varepsilon \phi_{;2}Q, \qquad \qquad
 F^2\phi_{,2}=(P_{;2}-Q)+\phi_{;2}(P+\varepsilon Q_{;2}).
$$
 From \eqref{relation between P and Q}  and  \eqref{identity 1}, where $\mathcal{I}=0$, we infer that the previous two formulae are satisfied because they are identities. Hence,  \eqref{fuster equivelent} is satisfied which means $\overline{F}$ is Berwaldian.\\
$\textbf{(c)}\Longleftrightarrow \textbf{(a)}$ It follows directly from Theorem \ref{first theorem of F bar berwald 2}
\end{proof}
\begin{remark}
  Let $(M,F)$ be a conic pseudo-Riemannian surface. The Berwaldian property of $(M,\overline{F})$ under the anisotropic conformal transformation \eqref{the anisotropic conformal transformation} depends on the existence of a symmetric tensor $M^j_{ik}$ satisfies \eqref{fuster equivelent}. In Theorem \ref{theorem of existence symmetric tensor on M}, we have found an explicit expression of $M^j_{ik}$.  Now, assuming that the functions $P$ and $Q$  satisfy $$P +\varepsilon P_{;2;2}=0 \quad\text{ and
}\quad Q +\varepsilon Q_{;2;2}=0.$$ Thus, the tensor $M^i_{jk}$ given by \eqref{symmetric tensor on the base maifold first}  takes the form 
\begin{align*}\label{symmetric tensor on the base maifold second}
     M^i_{jk}=& \frac{1}{F^{2}}[2(P\ell^{i}+Qm^{i})\ell_{j}\ell_{k}+\{(P_{;2}-Q)\ell^{i}+(\varepsilon P+ Q_{;2})m^{i}\}(\ell_{j}m_{k}+\ell_{k}m_{j})\nonumber\\
 &+2\{-Q_{;2}\ell^{i}+\varepsilon P_{;2}m^{i}\}m_{j}m_{k}].
\end{align*}  
\end{remark}
In case of regular ($\mathcal{A}=TM_0$) positive definite Finsler metrics, any affine Berwald connection corresponds to  Levi-Civita connection of some Riemannian metric on the same manifold. This result is known as \lq \lq Szabó's metrization theorem" \cite{szabo}. Consequently, a Berwald space is Riemann metrizable if its  Berwald connection coincides with the Levi–Civita connection of some Riemannian metric on the same manifold.
\par However, Szabó's metrization theorem \textbf{does not} extend to conic pseudo-Finsler surfaces. Using an anisotropic conformal transformation of a  conic pseudo-Finsler surface, we provide a counterexample to Szabó's  theorem (see, Example \ref{not Riem met.}).  Furthermore, we derive a necessary and sufficient condition for $\overline{F}$ to be Berwald metrizable by $F$, as stated in the following proposition. 

\begin{proposition}\label{proposition of metrizable equivalent}
   Let $(M,F)$ be a conic pseudo-Riemannian surface anisotropically  conformally changed to conic pseudo-Berwald surface $(M,\overline{F})$. Then $\overline{F}$ is Riemann metrizable if and only if $$\phi_{,1}=\phi_{,2}=0.$$ 
\end{proposition}
\begin{proof}
Let $F$ be a conic pseudo-Riemannian metric  anisotropically  conformally changed to 
 conic a pseudo-Berwald metric $\overline{F}$. Since the difference between two the connection coefficients $\overline{G}^i_{jk}\text{ and } G^i_{jk}$ is a tensor, that is,  $\overline{G}^i_{jk}-G^i_{jk}=M^i_{jk}.$ 
In view of Theorem \ref{theorem of existence symmetric tensor on M}, the symmetric tensor $M^i_{jk}(x)$ of type $(1,2)$ is given by \eqref{symmetric tensor on the base maifold first}. The conic pseudo-Berwald $\overline{F}$ is Riemann metrizable by $F$ if and only if $M^i_{jk}=0$, i.e.,
    \begin{align}
        0=&(2P\ell^{i}+2Qm^{i})\ell_{j}\ell_{k}+\{(P_{;2}-Q)\ell^{i}+(\varepsilon P+ Q_{;2})m^{i}\} (\ell_{j}m_{k}+\ell_{k}m_{j}) \\
          &+\{(\varepsilon P+P_{;2;2}-2Q_{;2})\ell^{i}+(2\varepsilon P_{;2}+\varepsilon Q+ Q_{;2;2})m^{i}\}m_{j}m_{k},\label{T tensor vanishs}
    \end{align}
    which is equivalent to $P=Q=0.$
    This is because $P=Q=0$ obviously implies  $M^i_{jk}=0$. On the other hand, contracting \eqref{T tensor vanishs} by $\ell_i \ell^j\ell^k$ and $m^i\ell^j\ell^k$, we get $P=Q=0.$ 
    \par From \cite[Theorem 4.11]{first paper}, $P=Q=0$ is also equivalent to $\phi_{,1}=\phi_{,2}=0.$
\end{proof}
We provide the following example of a conic pseudo-Riemannian metric that  anisotropically conformally changed to a conic pseudo-Berwald and they have the same connection. In other words, the conic pseudo-Berwald is Riemann metrizable.  For the calculations of the following examples, a Maple's code is posted on\\
 \url{
https://github.com/salahelgendi/Maple-s-code-for-calculations-of-Examples-3.11-3.12}
\begin{example}\label{example 2}
Let $M=\mathbb{B}^2\subset\mathbb{R}^2$, $x \in M,\; y\in T_{x}\mathbb{B}^{2}\cong \mathbb{R}^{2},\; a=(a_{1},a_{2})\in\mathbb{R}^{2} $ and $a$ is a constant vector with $|a|<1$. Let   
$$\displaystyle{z^i=\frac{(1+\langle a,x\rangle)y^i-\langle a,y \rangle x^i}{\langle a,y \rangle}}.$$
Define the Finsler metric $F$ by
$$F=\frac{\langle a,y\rangle \sqrt{(z^1)^2+(z^2)^2}}{(1+\langle a,x\rangle)^2}. $$
The  spray coefficients are given by
$$
G^i=-\frac{\langle a,y \rangle}{1+\langle a,x \rangle}y^i.
$$
Now, let $\overline{ F}= e^{\phi} F=\dfrac{\langle a,y\rangle \sqrt{(z^1)^2+(z^2)^2}}{(1+\langle a,x\rangle)^2} \ \exp\left( \sqrt{(z^1)^2+(z^2)^2}\right), $
where $\phi=\sqrt{(z^1)^2+(z^2)^2}.$ \\
One can easily check that $\overline{ F}$  is a conic pseudo-Finsler metric and $\delta_{i}\phi=0$. Furthermore,
$$\overline{G}^i=G^i=-\frac{\langle a,y \rangle}{1+\langle a,x \rangle}\,y^i.$$
\end{example}

By making use of Berwald-Rund metric \cite{Bao}, we  find a counterexample to  Szabó's theorem, which show that not all conic pseudo-Berwald spaces are Riemann metrizable.
\begin{example}\label{not Riem met.}
Let $M=\mathbb{R}^2$ and the coordinates of $(x,y)$ in $\mathcal{A}$ can be written as $(x^1,x^2; y^1,y^2)$. Define the conic pseudo-Finsler metric $F$ on $M$ by 
$$F=\frac{1}{x_2}\sqrt{2y_1 y_2 x_2^2+c y_2^2},$$ where $c:=1-2x_1x_2+\sqrt{1-4x_1x_2}.$
The main scalar of $F$  vanishes ($\mathcal{I}=0$), i.e., it is a conic pseudo-Riemannian metric.
The anisotropic conformal factor is expressed as $$\phi=\frac{3}{2}\ln{(\frac{2x_2^2y_1+cy_2}{x_2^2y_2})}.$$
We get the conic pseudo-Finsler metric 
$$\overline{F}=e^{\phi}F=\sqrt{2}y_2\left(\frac{c}{2x_2^2}+\frac{y_1}{y_2}\right)^2,$$ where $\sigma-\phi_{;2}^2+\varepsilon\neq0$ (necessary and sufficient for $\overline{F}$ to be pseudo-Finsler metric \cite{first paper}). One can easily check that $\overline{F}$ satisfies Theorem $\emph{\ref{first theorem of F bar berwald}}$, that is, 
$$-3\varepsilon Q-3Q_{;2;2}+\varepsilon P_{;2}+P_{;2;2;2}=0,\;\quad  \;\;3P+3\varepsilon P_{;2;2}+\varepsilon Q_{;2}+Q_{;2;2;2}=0.$$ Consequently,  $\overline{F}$ is a conic pseudo-Berwald metric. From Proposition \ref{proposition of metrizable equivalent}, the conic pseudo-Berwald metric $\overline{F}$ is not Riemann metrizable by $F$ because $\phi_{,1}$ and $\phi_{,2}$ do not vanish.   
\end{example}

\section{ Landsberg Surface and T-Condition}
Given a Finsler manifold, the $T$-condition refers to the vanishing of the $T$-tensor. In particular, for a Finsler surface,  the $T$-condition  and $\sigma$-condition (when Finsler space $(M, F)$  admits a non-constant function $\sigma(x)$ such that $(\partial_{i}\sigma)T^i_{jkr}=0$) are equivalent \cite{elgendi T}.

Suppose that  $(M,F)$ is a conic pseudo-Finsler surface anisotropically conformally changed to $(M,\overline{F})$. From \eqref{transform of m^i} and \eqref{I;b bar formula}, the $\overline{T}$-tensor ($\overline{F}\;\overline{T}_{ijkh}:=\mathcal{\overline{I}}_{;\,b}\,\overline{m}_i\overline{m}_j\overline{m}_k\overline{m}$) of $(M,\overline{F})$ can be expressed as follows
\begin{align}\label{T tensor of F bar}
\overline{F}\;\overline{T}_{ijkh}
=&e^{4\phi}\,(\frac{\varepsilon}{\rho})^{\frac{3}{2}}\,\overline{\mathcal{I}}_{;2}\,m_im_jm_hm_k.
\end{align}
\begin{remark}
\begin{description}
\item[(a)] From \eqref{T tensor of F bar},  the Finsler metric $\overline{F}$ satisfies the $\overline{T}$-condition if and only if $\mathcal{\overline{I}}_{;2}=0.$ 
\item[(b)]From \eqref{I; 2 bar formula} and \eqref{T tensor of F}, we get
\begin{align}\label{T tensor 2 of F bar}
\overline{T}_{ijhk}=&\frac{\varepsilon e^{3\phi}}{ \rho}\left[ T_{ijkh}+\frac{1}{2F\rho}\left(4\varepsilon\rho\phi_{;2;2}+\rho_{;2}(\mathcal{I}+2\varepsilon\phi_{;2}+\frac{\varepsilon\rho_{;2}}{2\rho})-\varepsilon\rho_{;2;2}\right) m_im_jm_hm_k\right].
\end{align}
\end{description}
\end{remark}

\begin{proposition}
Under the anisotropic conformal transformation \eqref{the anisotropic conformal transformation}, we have 
\begin{description}
\item[(a)]If $\phi_{;2}$ is an isotropic function then the $T$-condition is preserved.
\item[(b)]Suppose that $\mathfrak{\overline{g}}=e^{4\phi}\mathfrak{g}$ then the $T$-condition is preserved if and only if $\phi_{;2;2}=0.$ 
\end{description}  
\end{proposition}
\begin{proof}
\begin{description}
\item[(a)]If the Finsler metric $F$ satisfies the $T$-condition and $\phi_{;2}$ is an isotropic function, then
\begin{align}\label{isotropic I and phi_;2}
\phi_{;2;2}=\mathcal{I}_{;2}=0.
\end{align}
Consequently, from \eqref{formula of rho}, we deduce $\rho_{;2}=\rho_{;2;2}=0.
$
Hence,  $\overline{T}_{ijkh}=0$. 
\item[(b)] Since $\mathfrak{\overline{g}}=e^{4\phi}\mathfrak{g}$ is equivalent to $\sigma=(\phi_{;2})^2$ (by Remark \ref{remark rho_,1 and rho_,2} \textbf{(d)}). We get $\rho=\varepsilon$, by \eqref{formula of rho}.  
Thereby, $\rho_{;2}=\rho_{;2;2}=0$. Now, suppose $F$ satisfies $T$-condition. Thus, from \eqref{T tensor 2 of F bar}, we get  $\overline{F}$ satisfies $\overline{T}$-condition if and only if $\phi_{;2;2}=0.$
\end{description}\vspace*{-1.1cm}\[\qedhere\]
\end{proof}


\begin{proposition} \label{proposition of transform of Landsberg scalar}
Suppose that \eqref{the anisotropic conformal transformation} is the anisotropic conformal change of the two-dimensional conic pseudo-Finsler metric $F$. The Landsberg scalar $\overline{\mathcal{J}} $ of $\overline{F}$ is given by
\begin{align}\label{transform of landsberg scalar 1}
  F\overline{\mathcal{J}} =F\sqrt{\varepsilon\rho}\mathcal{J}+&\frac{\sqrt{\varepsilon\rho}}{2\rho}[( F^{2}\rho_{,1}-2\varepsilon Q\rho_{;2})(\mathcal{I}+2\varepsilon\phi_{;2}+\frac{\varepsilon\rho_{;2}}{2\rho})+F^{2}(4\varepsilon\rho\phi_{;2,1}-\varepsilon\rho_{;2,1})\nonumber \\
&-2Q(2\varepsilon\rho\mathcal{I}_{;2}+4\rho\phi_{;2;2}-\rho_{;2;2})].
\end{align}
\end{proposition}
\begin{proof}
The Landsberg scalar  $ \overline{\mathcal{J}} :=\overline{S}(\overline{\mathcal{I}})  =S(\overline{\mathcal{I}})-2\overline{G}^{i}\dot{\partial_{i}}\overline{\mathcal{I}}\overset{\eqref{transform of coefficient spray}}{=}S(\overline{\mathcal{I}})-2P\ell^{i}\dot{\partial_{i}}\overline{\mathcal{I}}-2Q \,m^{i}\dot{\partial_{i}}\overline{\mathcal{I}}.$\\
  From the homogeneity of $\overline{\mathcal{I}}$ and \eqref{transform of main scalar 2}  we get,
\begin{align}\label{Landsberge scalar of F bar of general form}
  F\overline{\mathcal{J}}& =FS(\overline{\mathcal{I}})-2\varepsilon\, Q\; {\overline{\mathcal{I}}}_{;2}=F^2\,\overline{\mathcal{I}}_{,1} -2\varepsilon \, Q\, {\overline{\mathcal{I}}}_{;2}.  
\end{align}
Hence, the result follows directly from \eqref{I; 2 bar formula} and \eqref{I,1 bar formula}. 
\end{proof}

\begin{remark}
\begin{description}
\item[(a)]  One can obtain an equivalent  formula of $\mathcal{\overline{J}}$ by using \eqref{transform of main scalar} and \eqref{Landsberge scalar of F bar of general form}, that is, 
\begin{align}\label{transform of landsberg scalar 2}
F\overline{\mathcal{J}} =&\frac{3\varepsilon}{2}(\varepsilon\rho)^{\frac{1}{2}}( F^{2}\rho_{,1}-2\varepsilon Q\rho_{;2})\left(  \mathcal{I}(1+\varepsilon \sigma)+\frac{1}{2}\sigma_{;2}+\phi_{;2}(\sigma+2\varepsilon)\right)  \nonumber \\
   &+(\varepsilon\rho)^{\frac{3}{2}}[(1+\varepsilon \sigma)(F^{2}\mathcal{I}_{,1}  -2\varepsilon Q\mathcal{I}_{;2})+(\varepsilon\mathcal{I}+\phi_{;2})(F^{2}\sigma_{,1}-2\varepsilon Q\sigma_{;2})\nonumber\\
   &+\frac{1}{2}(F^{2}\sigma_{;2,1}-2\varepsilon Q\sigma_{;2;2})+(\sigma+2\varepsilon)(F^{2}\phi_{;2,1}-2\varepsilon Q\phi_{;2;2})].
\end{align}
\item[(b)]We can get \eqref{I,a bar formula} directly from \eqref{Landsberge scalar of F bar of general form},   where $\overline{F}\;\overline{\mathcal{I}}_{,\,a}=\overline{\mathcal{J}}.$  
\end{description}
\end{remark}

\begin{theorem}
Consider the anisotropic conformal change \eqref{the anisotropic conformal transformation}  of the conic pseudo-Finsler surface $(M,F)$ and the main scalar $\overline{\mathcal{I}}$ is a first integral of the geodesic $S$ spray. Then  $\overline{F}$ is a Landsberg metric if and only if either $\phi_{;2}\phi_{,1}+\phi_{,1;2}-2\phi_{,2}=0$ or $\overline{F}$ is Berwaldian.
\end{theorem}
\begin{proof}
 From \eqref{I,a bar formula},  $\overline{F}$ is a Landsberg metric if and only if $\overline{\mathcal{I}}_{,1}=\frac{2\varepsilon}{F^2} Q\overline{\mathcal{I}}_{;2}.$
Let $\overline{\mathcal{I}}_{,1}=0,$  that is,    $\overline{F}$ is  Landsbergian if and only if either $Q=0$ or  $\overline{\mathcal{I}}_{;2}=0$. From \eqref{formula of Q only} and \eqref{second comutation}, $\overline{F}$ is a Landsbergian metric if and only if either $\phi_{;2}\phi_{,1}+\phi_{,1;2}-2\phi_{,2}=0$ or $\overline{\mathcal{I}}_{,2}=0$ . The latter gives $0=\overline{\mathcal{I}}_{,1}=\overline{\mathcal{I}}_{;2}=\overline{\mathcal{I}}_{,2}$ which means, by \eqref{I,a bar formula} and \eqref{I,b bar formula}, that $\overline{F}$ is Berwaldian.  
\end{proof}

\begin{proposition}\label{ special rho=varepslon}
Suppose that $(M,F)$ is a conic pseudo-Finsler surface anisotropically conformallt transformed by \eqref{the anisotropic conformal transformation}  with $\sigma=(\phi_{;2})^{2}$ to $(M,\overline{F})$. Then, we get
\begin{description}
\item[(a)] $F\overline{\mathcal{J}}=F\mathcal{J}+2\varepsilon F^2\phi_{;2,1}-2\varepsilon Q(\mathcal{I}_{;2}+2\varepsilon\phi_{;2;2}).$
\item[(b)] If $\phi$ is the first integral of the geodesic spary $S$, then the property of being Landsbergian is preserved if and only if either $\phi$ is horizontally constant  or $\mathcal{I}_{;2}+2\varepsilon\phi_{;2;2}=1$.  
\end{description}
\end{proposition}
\begin{proof}Let  $\sigma=(\phi_{;2})^{2}$. From \eqref{formula of rho}, we get 
$\rho=\varepsilon \text{ and } \phi_{;2;2}+\varepsilon\mathcal{I}\phi_{;2}+(\phi_{;2})^{2}=0.$
\begin{description}
\item[(a)] It follows from \eqref{transform of landsberg scalar 1}.
\item[(b)] If $\phi_{,1}=0$, by  \eqref{formula of Q only}, we get $ Q=-F^2\phi_{,2}.$ Thus, \textbf{(a)} becomes
\begin{align*}
\overline{\mathcal{J}}=\mathcal{J}+2\varepsilon F[\phi_{;2,1}+\phi_{,2}(\mathcal{I}_{;2}+2\varepsilon\phi_{;2;2})].
\end{align*} 
Applying \eqref{second comutation} for $f=\phi $, we get $\phi_{;2,1}=-\phi_{,2}$, which implies
\begin{align*}
\overline{\mathcal{J}}=\mathcal{J}+2\varepsilon F\phi_{,2}[\mathcal{I}_{;2}+2\varepsilon\phi_{;2;2}-1].
\end{align*} 
Then  the property of being Landsbergian is preserved if and only if either $\mathcal{I}_{;2}+2\varepsilon\phi_{;2;2}=1$ or $\phi_{,2}=0$. Clearly, $\phi_{,1}=0$ together with $\phi_{,2}=0$ gives $\phi$ is horizontally constant.
\end{description}
\vspace*{-0.6cm}\[\qedhere\]
\end{proof}

\begin{proposition}\label{riemannian to Landsberg}
Assume that \eqref{the anisotropic conformal transformation} is the anisotropic conformal change of conic pseudo  Riemannian metric $F$  with $\phi_{;2}$ being horizontally constant. The Finsler metric $\overline {F}$ is Landsberg if and only if either  $\overline{T}$-tensor vanishes identically or  $\phi_{;2}\phi_{,1}=\phi_{,2}.$
\end{proposition}
\begin{proof}
From Corollary \ref{proposition phi_;2  d_h closed and F riemannian} and \eqref{I,1 bar formula}, we obtain
\begin{equation}\label{I bar ,1=0}
\mathcal{\overline{I}}_{,1}=0.
\end{equation}  
Applying \eqref{second comutation} for $f=\phi$, by \eqref{formula of Q only}, we get $2Q=\varepsilon\rho F^2(\phi_{;2}\phi_{,1}+\phi_{;2,1}-\phi_{,2})$, but $\phi_{;2}$ is horizontally constant, then
\begin{align}\label{1st relarion in proposition of F landesberg}
2Q&=\varepsilon\rho F^2(\phi_{;2}\phi_{,1}-\phi_{,2}).
\end{align} 
Plugging \eqref{I bar ,1=0} and  \eqref{1st relarion in proposition of F landesberg} into \eqref{I,a bar formula}, we have $\overline{\mathcal{I}}_{,\,a}= - e^{-\phi}\,\rho\, \overline{\mathcal{I}}_{;2}\,(\phi_{;2}\,\phi_{,1}-\phi_{,2}).$
Hence, $\overline{F}$ is  Landsbergian ($\overline{\mathcal{I}}_{,a}=0$)  if and only if  either $\overline{F}$ satisfies $\overline {T}$-condition $(\overline{\mathcal{I}}_{;2}=0)$ or  $\phi_{;2}\,\phi_{,1}=\phi_{,2}.$ 
\end{proof}
\begin{proposition}\label{in view of unicorn}
 Assume that $(M,F)$ is anisotropically conformally changed to $(M,\overline{F})$. The Finsler metric $\overline{F} $ is  Berwaldian if one of the following conditions is satisfied:
\begin{description}
\item[(a)]  $\overline{F} $ has an isotropic main scalar and $\overline{\mathcal{I}}_{,1}=0$. 
\item[(b)]   $\overline{F}$ is  Landsbergian provided that $F^2\overline{\mathcal{I}}_{,2}=( P+\varepsilon Q_{;2}-\mathcal{I}Q)\overline{\mathcal{I}}_{;2}.$
\end{description}
\end{proposition}
\begin{proof}
\begin{description}
\item[(a)] The fact that $\overline{F} $ having an isotropic main scalar $\overline{\mathcal{I}}$ implies $\overline{\mathcal{I}}_{,a}=0$ and $\overline{\mathcal{I}}_{,b}=\overline{\mathcal{I}}_{,2}$.\\
Applying \eqref{second comutation} for $f=\overline{\mathcal{I}}$, we get $\overline{\mathcal{I}}_{,2}=0$. Hence, $\overline{\mathcal{I}}_{,a}=0$ and $\overline{\mathcal{I}}_{,b}=0$ which means $\overline{F}$ is Berwaldian.
\item[(b)] Since $\overline{F}$ is  Landsbergian, it is equivalent to $\overline{\mathcal{I}}_{,a}=0$, then by \eqref{I,a bar formula} 
\begin{equation}\label{equivelent to landsberg 2}
\overline{\mathcal{I}}_{,1}=\frac{2\varepsilon}{F^2}Q\overline{\mathcal{I}}_{;2}.
\end{equation} 
Plugging \eqref{equivelent to landsberg 2} into \eqref{I,b bar formula}, we get
\begin{equation}\label{equation 3.8}
\overline{\mathcal{I}}_{,b}=\overline{\mathcal{I}}_{,2}-\frac{2\varepsilon}{F^2}\phi_{;2}Q\overline{\mathcal{I}}_{;2}-\frac{\varepsilon}{F^2}(\varepsilon P+Q_{;2}-\varepsilon\mathcal{I}Q-2\phi_{;2}Q)\overline{\mathcal{I}}_{;2}.
\end{equation}
As $ F^2\overline{\mathcal{I}}_{,2}=( P+\varepsilon Q_{;2}-\mathcal{I}Q)\overline{\mathcal{I}}_{;2}$ (by assumption), then $\overline{\mathcal{I}}_{,b}=0$. Hence, $\overline{F}$ is Berwaldian.
\end{description}
\vspace{-0.5cm}{\qedhere}
\end{proof}
\section{ Anisotropically Conformally Flat and  Douglas Surfaces}
\subsection{ Locally Minkowski surfaces}
\begin{definition}\emph{\cite{chernbook2000}}
A Finsler space is called  locally Minkowski metric if, at each point $x \in M,$ there exists local coordinates $(x^{i},y^{i})$ on $TM$ such that $F$ is a function of $y$ only, in addition,   $G^i=0$. This coordinate system is called adopted coordinate system. 
\end{definition}

In this subsection, we work in the adopted coordinate system with out mentioning that. 
\begin{definition}
A  pseudo-Finsler manifold $(M,F)$ is said to be anisotropically conformally flat if $F$ is anisotropically conformally related to a locally Minkowski metric.
\end{definition}

\begin{proposition}\label{theorem anisotropic conf flat}
Let the conic pseudo-Finsler surface $(M,F)$ be anisotropically conformally changed to  $(M,\overline{F})$. Then, the following properties  are equivalent:     
\begin{description}\label{proposition of anisotropic flat}
    \item[(a)]$F \partial_{j}\phi+\partial_{j}F=0.$
    \item[(b)] $F$ is anisotropically conformally flat.
    \item[(c)]  $P=-\frac{1}{2}y^r(\partial_r F)$ and $Q=-\frac{1}{2}\frac{F^2(\dot{\partial_2}\partial_1 F-\dot{\partial_1}\partial_2 F)}{h}.$
\end{description}
\end{proposition}
\begin{proof}
$\textbf{(a)}\Longleftrightarrow\textbf{(b)}$
Under the anisotropic transformation $\overline{F}=e^{\phi}F$, we have
\begin{equation}
\partial_i \overline{F}=e^\phi(F\partial_i\phi+\partial_iF).
\end{equation}
Hence, $F$ being anisotropically conformally flat is equivalent to $F \partial_{j}\phi+\partial_{j}F=0.$\\
$\textbf{(b)}\Longleftrightarrow\textbf{(c)}$ By \eqref{geodesic of 2 dim} and  \eqref{transform of coefficient spray}, we obtain 
\begin{equation*}
\overline{G}^i=G^i+Qm^i+P\ell^i=(\frac{1}{2} y^r \partial_r F+P)\ell^i+(\frac{F^2(\dot{\partial_2}\partial_1 F-\dot{\partial_1}\partial_2 F)}{2h}+Q)m^i .
\end{equation*}
The Finsler metric $F$ is anisotropically flat if and only if $\overline{G}^i =0$, that is, 
\begin{equation}\label{spray and anisotropically flat}
0=(\frac{1}{2} y^r \partial_r F+P)\ell^i+(\frac{F^2(\dot{\partial_2}\partial_1 F-\dot{\partial_1}\partial_2 F)}{2h}+Q)m^i .
\end{equation}
Multiple \eqref{spray and anisotropically flat} by $\ell_i$ and $m_i$ we get  $y^r(\partial_r F)+2P=0$ and $\frac{F^2(\dot{\partial_2}\partial_1 F-\dot{\partial_1}\partial_2 F)}{h}+2Q=0.$ Hence, $F$ being anisotropically conformally flat is equivalent to $P=-\frac{1}{2}y^r(\partial_r F)$ and $Q=-\frac{1}{2}\frac{F^2(\dot{\partial_2}\partial_1 F-\dot{\partial_1}\partial_2 F)}{h}.$
\end{proof}


\begin{corollary}\label{proposition anisotropic conf flat}
Given that $(M,F)$ is a locally Minkowskian surface. Then, $F$ is anisotropically conformally flat under the anisotropic change \eqref{the anisotropic conformal transformation}  if and only if the  conformal  factor is either a function of $y$ alone or  homothetic.
\end{corollary}
\begin{proof}
By Proposition \ref{theorem anisotropic conf flat}, the Finsler metric $\overline{F}$ is locally Minkowskian surface if and only if $$F \partial_{j}\phi+\partial_{j}F=0.$$ Since  $F$ is  a locally Minkowskian metric ($\partial_{j}F=0$), then $F$ is anisotropic conformally flat if and only if the anisotropic conformal factor is either function of $y$ only or  homothetic.
\end{proof}

\begin{proposition}\label{prop proj. equ. Minko}
Let $F$ be an anisotropic conformally flat conic pseudo-Finsler metric. If $F^2\phi_{,1}+y^k\,\partial_{k}F=0$, then F is either projectively flat or $\phi$ is an isotropic function on some coordinate system of $TM$. 
\end{proposition}
\begin{proof}
Let $F$ be  anisotropic conformally flat. By Proposition \ref{theorem anisotropic conf flat},  we have $\ell^k(F \partial_{k}\phi+\partial_{k}F)=0.$ 
From \eqref{eq}, we get 
\begin{equation}\label{necessary for conformally flat}
F^2\phi_{,1}+2 \phi_{;2}G^k\;m_k+2\,G^k\,\ell_{k}=0. 
\end{equation}
\par Assume that  $F^2\phi_{,1}+y^k\,\partial_{k}F=0,$ which is equivalent to $ F^2\phi_{,1}+2G^k\ell_{k}=0,$
by \eqref{eq}. Substituting into \eqref{necessary for conformally flat}, we obtain $\phi_{;2}\,G^k \,m_k=0. $      
Therefore,  $F$ is either projectively flat ($G^k \,m_k=0$) or $\phi$ is an isotropic function ($\phi_{;2}=0$).  
\end{proof}
In view proposition \ref{prop proj. equ. Minko}, if F is anisotropically conformally flat, then $\overline{F}$ is projectively flat in some coordinate system of $TM$, namely, the adopted coordinate system. Morover, the condition $F^2\phi_{,1}+2 G^k\phi_{;2}\;m_k+2\,G^k\,\ell_{k}=0$ is necessary but not sufficient for a conic pseudo-Finsler F to be anisotropically conformally flat.
\subsection{Douglas surfaces}
Since, the Douglas tensor is an invariant under projective change, we end our results by investigating it under the anisotropic conformal change \eqref{the anisotropic conformal transformation}.
\begin{definition}
A pseudo-Finsler metric is of Douglas type, if its Douglas tensor vanishes identically or equivalently \emph{\cite{Bacs-Matsumoto 2}} if the functions $D^{ij}:=G^iy^j-G^jy^i$ are homogeneous polynomials of degree three in $(y^i).$  
\end{definition}
By using \eqref{transform of ell^i} and \eqref{transform of coefficient spray}, we obtain
\begin{align}\label{transform of D^ij}
F\overline{D}^{ij}=FD^{ij}+Q(m^i\ell^j-m^j\ell^i).
\end{align}  
 It is clear from \eqref{transform of D^ij}, the property of $F$ being a Douglas metric is preserved under the anisotropic conformal change \eqref{the anisotropic conformal transformation}, if and only if $Q=0$ (or equivalently $\phi_{;2}\phi_{,1}+\phi_{,1;2}-2\phi_{,2}=0$). Consequently, the anisotropic conformal transformation, with $\phi_{;2}\phi_{,1}+\phi_{,1;2}-2\phi_{,2}=0$, of a Riemann (or Berwald) metric is of Douglas type.
 
 Assume that $F$ is a projectively flat metric and $Q=0$, then $\overline{F}$ is projectively flat \cite{first paper}. Consequently, $\overline{F}$ is of Douglas type when  $F$ is projectively flat if and only if $3R_{,2}-R_{;2,1}=0$ and $F$ is a Douglas metric (i.e., $6\mathcal{I}_{,1}+\varepsilon\mathcal{I}_{2;2}+2\mathcal{I}\mathcal{I}_{2}=0$) \cite{Bacs-Matsumoto}. Hence we have the following: 
\begin{proposition}
Let \eqref{the anisotropic conformal transformation} be the anisotropic conformal transformation. Then $\overline{F}$ is a Douglas metric if one of the following holds
\begin{description}
\item[(a)] $F$ is Riemannian (or Berwald) metric and $Q=0.$
    \item[(b)] $F$ is projectively flat metric and $Q=0.$
\end{description}
\end{proposition}

\begin{theorem}\label{theorem necc and suff for F bar Douglas }
Assuming that $(M,F)$ is a conic pseudo-Berwald Finsler surface, \eqref{the anisotropic conformal transformation} is the anisotropic conformal transformation. The Finsler metric $\overline{F}$ is of Douglas type if and only if $ 3\psi=\varepsilon\chi_{;2}+\mathcal{I}\chi$, where $\psi$ and  $\chi$ are given by
{\small{\begin{align} \label{psi fu}
F^{2}\psi &=[ (-3+\mathcal{I}_{;2})\varepsilon\,Q+(1+\mathcal{I}_{;2}+2\varepsilon\mathcal{I}^{2})\varepsilon\,P_{;2}+P_{;2;2;2}-3Q_{;2;2} +3\varepsilon\mathcal{I}P_{;2;2}-3\varepsilon\mathcal{I}Q_{;2}+2\mathcal{I}P],\end{align}}}
\begin{align}
F^{2}\chi &=[3P-(\mathcal{I}+\varepsilon\mathcal{I}_{;2;2}+\mathcal{I}\mathcal{I}_{;2})Q-(2\varepsilon\mathcal{I}_{;2}+\mathcal{I}^{2}-\varepsilon) Q_{;2}+3\varepsilon P_{;2;2}+Q_{;2;2;2}+3\mathcal{I}P_{;2}].\label{chi fun}
\end{align}
\end{theorem}
\begin{proof}
Let $(M,F)$ be a conic pseudo-Finsler surface. The Berwald curvature written as  
\begin{align}\label{berwald curvature for F}
FB^{i}_{jkr}=[-2\mathcal{I}_{,1}\ell^{i}+\mathcal{I}_{2}m^{i}]m_{j}m_{k}m_{r}.
\end{align}
Under the anisotropic conformal transformation \eqref{the anisotropic conformal transformation}, the Berwald curvature of $\overline{F}$ is given by
\begin{align*}
\overline{F}\;\overline{B}^{i}_{jkr}=[-2\overline{\mathcal{I}}_{,\,a}\overline{\ell}^{i}+\overline{\mathcal{I}}_{d}\overline{m}^i]\overline{m}_{j}\overline{m}_{k}\overline{m}_{r}.
\end{align*}
From \eqref{transform of ell^i} and \eqref{transform of m^i}, we get
\begin{align}\label{berwald curvature of F bar after tans berwald frame}
F\overline{B}^{i}_{jkr}&=e^{\phi}[-2\overline{\mathcal{I}}_{,\,a}\ell^{i}+\sqrt{\varepsilon\rho}\overline{\mathcal{I}}_{d}(m^{i}-\varepsilon\phi_{;2}\ell^{i})](\frac{\varepsilon}{\rho})^{\frac{3}{2}} m_{j}m_{k}m_{r}\nonumber\\
&=e^{\phi}(\frac{\varepsilon}{\rho})^{\frac{3}{2}}[\{-2\overline{\mathcal{I}}_{,\,a}-\varepsilon\sqrt{\varepsilon\rho}\phi_{;2}\overline{\mathcal{I}}_{d}\}\ell^{i}+\sqrt{\varepsilon\rho}\overline{\mathcal{I}}_{d}m^{i}]m_{j}m_{k}m_{r}.
\end{align}
From \eqref{transform of berwald curvature of spray derivative} and \eqref{berwald curvature for F} we get
\begin{equation}\label{berwald curvature of F bar with chi, psi}
F\overline{B}^{i}_{jkr}=[(-2\mathcal{I}_{,1}+\psi)\ell^{i}+(\mathcal{I}_{2}+\chi) m^{i}]m_{j}m_{k}m_{r}.
\end{equation}
Then, we have from \eqref{berwald curvature of F bar after tans berwald frame} and \eqref{berwald curvature of F bar with chi, psi} 
\begin{align}
\overline{\mathcal{I}}_{d}=&\varepsilon e^{-\phi}\rho (\mathcal{I}_{2}+\chi), \qquad
\overline{\mathcal{I}}_{,\,a}=\frac{-1}{2}e^{-\phi}(\varepsilon\rho)^{\frac{3}{2}}[-2\mathcal{I}_{,1}+\psi+\varepsilon\phi_{;2}(\mathcal{I}_{2}+\chi)], \label{I,a with chi and psi} 
\end{align}
where $\psi, \, \chi$ defined by \eqref{psi fu} and \eqref{chi fun}.
 Since $F$ is  Berwaldian $(\mathcal{I}_{,1}=\mathcal{I}_2=0)$,  \eqref{I,a with chi and psi}  has the form:
\begin{align}
\overline{\mathcal{I}}_{d}=&\varepsilon e^{-\phi}\rho \chi, \qquad \overline{\mathcal{I}}_{,\,a}=\frac{-1}{2}e^{-\phi}(\varepsilon\rho)^{\frac{3}{2}}[\psi+\varepsilon\phi_{;2}\chi].\label{second equ of theorem equ necess for Douglas}
\end{align}
From \eqref{I;b bar formula} and the first equation of \eqref{second equ of theorem equ necess for Douglas}, we get
\begin{align}\label{third equ of theorem equ necess for Douglas}
\overline{\mathcal{I}}_{d;b}=&\sqrt{\varepsilon\rho}\,\overline{\mathcal{I}}_{d;2}=\varepsilon e^{-\phi}\sqrt{\varepsilon\rho}[-\phi_{;2}\rho \chi+\rho_{;2} \chi+\rho \chi_{;2}]
\end{align}
From \eqref{transform of main scalar 2}, \eqref{second equ of theorem equ necess for Douglas} and \eqref{third equ of theorem equ necess for Douglas}, the Finsler metric $\overline{F}$ is a Douglas metric if and only if
\begin{align*}
0=&6\overline{\mathcal{I}}_{,\,a}+\varepsilon\overline{\mathcal{I}}_{d;b}+2\overline{\mathcal{I}}\;\overline{\mathcal{I}}_d\\
=& e^{-\phi}(\varepsilon\rho)^{\frac{3}{2}}[-3\psi-4\varepsilon\phi_{;2}\chi+\varepsilon\frac{\rho_{;2}}{\rho}\chi+\varepsilon\chi_{;2}]+2e^{-\phi}(\varepsilon\rho)^{\frac{3}{2}}[\mathcal{I}\chi+\varepsilon\phi_{;2}\chi-\frac{\varepsilon\rho_{;2}}{2\rho}\chi]\\
=& e^{-\phi}(\varepsilon\rho)^{\frac{3}{2}}[\mathcal{I}\chi-3\psi+\varepsilon\chi_{;2}].
\end{align*}
\vspace*{-1.3cm}\[\qedhere\]
\end{proof}

\begin{corollary}\label{proposition necc suff F bar Douglas}
If $(M,F)$ is a conic pseudo-Riemannian surface, then   $\overline{F}$ is of Douglas type if and only if $9\varepsilon Q+10Q_{;2;2}+\varepsilon Q_{;2;2;2;2}=0$.
\end{corollary}
\begin{proof} Since $I =0$, by assumption. 
From Theorem \ref{theorem necc and suff for F bar Douglas },  $\overline{F}$ is of Douglas type if and only if $ 3\psi=\varepsilon\chi_{;2}.$ Setting $I=0$ in \eqref{psi fu}, \eqref{chi fun}, the proof follows.
\end{proof}
\section{Conclusion}
We complete what we started in \cite{first paper} by developing the framework of anisotropic conformal transformations of Finsler surfaces, exploring the relationships between important Finslerian geometric objects such as Berwald, Landsberg, and Douglas tensors. We emphasize on how an anisotropic conformal transformation can convert a pseudo-Riemannian metric into non-Riemannian pseudo-Finsler metric which contrasts  isotropic conformal transformations.\\
The following points are to be highlighted:
\begin{itemize}
    \item The v-scalar derivatives $(f_{;\,a}, f_{;b})$ and h-scalar derivatives $(f_{,a}, f_{,b})$ has been defined in $(M,\overline{F})$ for a scalar field $f \in C^{\infty}(TM_0)$. Consequently, we find $\mathcal{\overline{I}}_{;b},\,\mathcal{\overline{I}}_{,a}$ and $\mathcal{\overline{I}}_{,b}$ which characterize all special conic pseudo Finsler surfaces associated with $\overline{F}$ to be either Berwaldian or Landsbergian or Douglasian. We have found out the anisotropic conformal transformation of Berwald curvature and Landsberg scalar.
   
    \item We have found two equivalent conditions for  a Riemannian metric to be anisotropically conformal transformed to a Berwald metric. The first is a relation between $P$, $Q$ and their derivatives. The second is an expression of the symmetric tensor $M^j_{ik}$ on the base manifold  satisfies $ F^2\delta_{i}(e^{2\phi})=y^kM^j_{ik}\dot{\partial}_{j}\overline{F}^2.$ 
    \item  The conic pseudo-Finsler metric $\overline{F}$ satisfies the $T$-condition if and only if its main scalar is isotropic with respect to $F$. Additionally, under the assumption that the determinant of the Finsler metric tensor is  invariant, we find an equivalent condition such that the $T$-condition is preserved under the anisotropic conformal transformation.
    \item Under the anisotropic conformal transformation, we have found necessary and sufficient  conditions for $\overline{F}$ to be Landsbergian. Moreover, we found the conditions for the Landsberg metric $\overline{F}$ to be Berwaldian.
    \item If the conformal factor depends on position alone, the anisotropic conformal transformation reduced to the well-known  isotropic conformal transformation. Consequently, the results obtained align precisely with those derived from the isotropic conformal transformation. For example, in this case,  we have:
\begin{align}\label{all variable when phi of x only}
\phi_{;2}=0, \quad  \sigma=0, \quad \rho=\varepsilon, \quad \overline{\mathcal{I}}=\mathcal{I}, \quad Q=-\frac{1}{2} F^2\phi_{,2},\quad P=\frac{1}{2}F^2\phi_{,1}.
\end{align}
-- The property of being Berwaldian is preserved under the anisotropic conformal transformation  if  either the conformal factor is homothetic or  $F$ satisfies the $T$-condition.
 \par-- From \eqref{all variable when phi of x only} in \eqref{transform of landsberg scalar 1}, we get 
 $ \overline{\mathcal{J}}=\mathcal{J}+2F\phi_{,2} \mathcal{I}_{;2}.$
    We have $\overline{\mathcal{J}}=\mathcal{J}$ if and only if 
   either $\phi_{,2}=0 $ or $\mathcal{I}_{;2}=0$.
   If $\phi_{,2}=0$, then by \eqref{third comutation} and $\phi_{;2}=0$,  we get  $\phi$ is a constant function, that is, $\phi$ is homothetic (trivial case). Consequently, $ \overline{\mathcal{J}}=\mathcal{J}$ if $\mathcal{I}_{;2}=0$ ($F$ satisfies $T$-condition) which is also mentioned in \cite[Theorem 3.9]{elgendi T}.
\par   
-- Now, plugging \eqref{all variable when phi of x only} into \eqref{I,a bar formula} and \eqref{I,b bar formula}, respectively, we obtain 
\begin{align}\label{1st prime berwald when phi of x only}\overline{\mathcal{I}}_{,\,a}&=e^{-\phi} [\mathcal{I}_{,1}+\varepsilon \phi_{,2}\mathcal{I}_{;2}], \qquad
\overline{\mathcal{I}}_{,\,b}=e^{-\phi} [\mathcal{I}_{,2}-( \phi_{,1}+\mathcal{I}\phi_{,2})\mathcal{I}_{;2}].
\end{align}
Applying \eqref{third comutation} for $f=\phi$, we get 
$0=-\varepsilon(\phi_{,1}+\mathcal{I}\phi_{,2}) $.
Thus, $\overline{\mathcal{I}}_{,\,b}=e^{-\phi} [\mathcal{I}_{,2}].$ Therefore, the property of Berwaldian is preserved if and only if either $F$ satisfies the $T$-condition or the conformal factor is homothetic which is also mentioned in \cite[Theorem 3.11]{elgendi T}.

\end{itemize}

\end{document}